\documentclass[12pt]{amsart} 
 \normalfont

\usepackage[psamsfonts]{amssymb}
\usepackage{pifont}
\usepackage{tikz}
\theoremstyle{plain}

\input{xy}
\xyoption{all}

\newtheorem{theorem}{Theorem}[section]
\newtheorem{lemma}[theorem]{Lemma}
\newtheorem{proposition}[theorem]{Proposition}
\newtheorem{corollary}[theorem]{Corollary}

\def\bc{\mathbb{C}}

\def\br{\mathbb{R}}

\def\fA{\mathfrak{A}}

\def\fB{\mathfrak{B}}

\def\fU{\mathfrak{U}}
\def\fG{\mathfrak{G}}
\def\fH{\mathfrak{H}}
\def\fI{\mathfrak{I}}

\def\fK{\mathfrak{K}}

\def\fP{\mathfrak{P}}

\def\fU{\mathfrak{U}}
\def\fu{\mathfrak{u}}

\def\fZ{\mathfrak{Z}}
\def\fs{\mathfrak{s}}
\def\fo{\mathfrak{o}}

\def\fp{\mathfrak{p}}

\usepackage{amsaddr}

\newcounter{commentlabel}
\begin{document}
  \title{Isometries of Clifford Algebras II}
\author{Patrick Eberlein}
\address{Department of Mathematics, University of North Carolina, Chapel Hill, NC 27599}
\email{pbe@email.unc.edu}
\date{\today}
\maketitle

\noindent $\mathbf{Abstract}$  Let F be a field of characteristic $\neq 2$, and let $F^{n}$ denote the vector space of n-tuples of elements in F.  Let $\{e_{1}, ... , e_{n} \}$ denote the canonical basis of $F^{n}$.  Let r and s be nonnegative integers such that $r+s = n$, and let Q denote the nondegenerate, symmetric, bilinear form on $F^{n}$ such that $Q(e_{i},e_{j}) = 0$ if $i \neq j, Q(e_{i},e_{i}) = 1$ if $1 \leq i \leq r$ and $Q(e_{r+j}, e_{r+j}) = - 1$ for $1 \leq j \leq s$.  Let $C\ell(r,s)$ denote the Clifford algebra determined by Q and $F^{n}$.  There is a canonical extension of Q to a nondegenerate, symmetric bilinear form $\overline{Q}$ on $C\ell(r,s)$.  An element g of $C\ell(r,s)$ will be  called an isometry of $C\ell(r,s)$ if left and right translations by g preserve $\overline{Q}$.  Let $G_{r,s}$ denote the group of all isometries of $C\ell(r,s)$.  We construct a Lie algebra $\fG_{r,s}$ over F that equals the Lie algebra of $G_{r,s}$ in the case that $F = \br$ or $\bc$.  The Lie algebra $\fG_{r,s}$ admits an involutive automorphism  whose $+1$ and $-1$ eigenspaces determine a Cartan decomposition $\fG_{r,s} = \fK_{r,s} \oplus \fP_{r,s}$.  We compute the bracket relations for a natural system of generators of $\fG_{r,s}$.  Finally, we determine $\fG_{r,s}$ in the case that $F = \br$.
\newline

\noindent $\mathbf{Remark}$  The main result of this preprint is known, but is left to the reader as a "hard exercise".  We would like to leave this preprint on the ArXiv in case the details of one solution are of interest to the reader.  
\newline

\noindent I was informed by the referee of "Isometries of Clifford Algebras II" that the main results of "Isometries of Clifford Algebras I" and Isometries of "Clifford Algebras, II" are both contained in the table labeled $R_{p,q}$ at the top of page 271 of I. Porteus, Topological Geometry, Cambridge University Press, Cambridge, 1969.  This result of Porteous also appears later in table XB on page 737 of the paper "Scalar products of spinors and an extension of Brauer-Wall groups" by P. Lounesto.  This paper was published in Foundations of Physics, vol. 11, Nos. 9/10, 721-740. 
\newline

\noindent $\mathbf{Introduction}$ This paper is a continuation of [E], in which we computed the compact isometry group $G_{n,0}$ in the case that $F = \br$.  The goal here is to compute the most general Clifford Lie algebra $\fG_{r,s}$ in the case that $F = \br$.
\newline 

\noindent In section 1 we state briefly some results from [E] that will be useful here.  In section 2 we define the Clifford Lie algebra $\fG_{r,s}$ and an involutive automorphism $\beta$ whose $+1$ and $-1$ eigenspaces determine a Cartan decomposition $\fG_{r,s} = \fK_{r,s} \oplus \fP_{r,s}$.  We compute the dimensions of $\fG_{r,s}, \fK_{r,s}$ and $\fP_{r,s}$.  The Lie algebra $\fG_{r,s}$ also admits a decomposition $\fG_{r,s} = \fZ(\fG_{r,s}) \oplus\fH_{r,s}$, where  $\fZ(\fG_{r,s})$ is the center of $\fG_{r,s}$ and $\fH_{r,s}$ is an ideal whose Killing form is nondegenerate.  The center $\fZ(\fG_{r,s})$ is 1-dimensional if $r+s \equiv 1~(mod~4)$ and $\{0 \}$ otherwise.  In section 3 we determine the bracket relations for a natural system of generators for $C\ell(r,s)$.  In section 4 we determine the center of $\fK_{r,s}$ in terms of these generators.  In section 5 we use the main result of [E] to determine the complexifications of the semisimple ideals $\fH_{r,s}$ in the case that $F = \br$.  In section 6 we develop a method for determining $\fH_{r,s}$ in the case that $F = \br$.  If $r+s \neq 3~(mod~4)$, then $\fH_{r,,s}$ is a real form of $\fU^{\bc}$, where $\fU$ is a simple, real Lie algebra whose Killing form is negative definite.  By work of E. Cartan all real forms of $\fU^{\bc}$ are determined by involutive automorphisms $\tau$ of $\fU$, and there are only two or three conjugacy classes in $Aut(\fU)$ of such automorphisms $\tau$.  Determining  the right conjugacy class involves computing the dimension of $\fK_{r,s}^{\prime} = \fK_{r,s} \cap \fH_{r,s}$ and the dimension of the center of  $\fK_{r,s}^{\prime}$.  If $r+s \equiv 3~(mod~4)$, then the method for determining $\fH_{r,s}$ is a more complicated version of the method just described.  In sections 7,8 and 9 we carry out the computation of $\fH_{r,s}$.
  
\section{Preliminaries}   

\noindent We list some results from [E] that will be useful here, with references from [E] in parentheses.  See also [Ha], [LM] and [FH]. 
\newline

\noindent Let V be a finite dimensional vector space over a field with characterstic $\neq 2$.  Let $Q : V \times V \rightarrow F$ be a non degenerate, symmetric bilinear form.  There exists a  basis $\{v_{1}, v_{2}, ... , v_{n} \} , n = dim~V$ of V that is Q-orthogonal ; that is, $Q(v_{i},v_{j}) = 0$ if $i \neq j$ and $Q(v_{i},v_{i}) \neq 0$ for all i. 

\begin{proposition}  (Proposition 1.1)  Let (V,Q) be as above.   Then there exists an F-algebra $C\ell(V,Q)$ and an injective linear map $i : V \rightarrow C\ell(V,Q)$ with the following property, which characterizes $C\ell(V,Q)$ up to algebra isomorphism :   Let $\fA$ be any associative finite dimensional  algebra over F, and let $\sigma : V \rightarrow \fA$ be any F - linear map such that $\sigma(v) \cdot \sigma(v) = - Q(v,v)1$ for all v $\in$ V.  Then there exists a unique algebra homomorphism $j : C\ell(V,Q) \rightarrow \fA$ such that $j \circ i = \sigma$.
\end{proposition}

\begin{proposition}  There is a canonical extension of Q to a nondegenerate, symmetric bilinear form $\overline{Q}$ on $C\ell(V,Q)$.  (See Proposition 4.1 for a precise statement)
\end{proposition}

\noindent A k-tuple $I = (i_{1}, ... , i_{k})$ is a multi-index if $i_{r} < i_{r+1}$ for $1 \leq r \leq k-1$.   Let $e_{I}$ denote $e_{i_{1}} \cdot e_{i_{2}} \cdot ... \cdot e_{i_{k}}$, and let $|I| = k$, the length of the multi-index I.
 
\begin{proposition} (Corollary 6.10 and Lemma 4.2) Let $\fI$ denote the set of all multi-indices.  Let $\fB = \{1, e_{I} : I \in \fI \}$. Then $\fB$ is a $\overline{Q}$-orthogonal basis of $C\ell(V,Q)$.  Moreover, $e_{I}^{2}$ is a nonzero element of F for all $I \in \fI$.
\end{proposition}

\begin{proposition}  (Section 1)  

	1)  There is a unique algebra automorphism $\alpha : C\ell(V,Q) \rightarrow C\ell(V,Q)$ such that $\alpha^{2} = Id$ on $C\ell(V,Q)$ and  $\alpha \equiv - Id$ on V.
	
	2)  There is a unique algebra anti-automorphism $c : C\ell(V,Q) \rightarrow C\ell(V,Q)$ such that $c^{2} = Id$ on $C\ell(V,Q)$ and $c \equiv - Id$ on V.
	
	3)  The maps $\alpha$ and c commute on $C\ell(V,Q)$.
\end{proposition}

\begin{proposition} (Lemmas 6.4 and 6.5)  Let $\fG$ denote the $-1$-eigenspace of the anti-automorphism $c : C\ell(V,Q) \rightarrow C\ell(V,Q)$.  Then 

	1)  $\fG$ is a Lie algebra over F.
	
	2)  $\fG = F-span~\{e_{I} : |I| \equiv 1~\rm{ or}~ 2~ (mod~4) \}$.
\end{proposition} 

\begin{proposition}  Let $n = dim~V$ and let $\{e_{1}, e_{2}, ... , e_{n} \}$ denote a Q-orthogonal basis of V.  Let $\omega = e_{1} \cdot e_{2} \cdot ... \cdot e_{n}$.  Then

	a) (Proposition 5.1)  $\omega$ lies in the center of $C\ell(V,Q) \Leftrightarrow$ n is odd.
	
	b)  (Corollary 6.3)  $\fZ(G)$, the center of $\fG$ equals $F \omega$ if $n \equiv 1~(mod~4)$ and $\fZ(\fG) = \{0 \}$ otherwise.
\end{proposition}

\begin{proposition}  (Proposition 6.12)There exists an ideal $\fH$ of $\fG$ such that $\fG = \fZ(\fG) \oplus \fH$ and the Killing form of $\fH$ is nondegenerate on $\fH$. 
\end{proposition}

\noindent $\mathbf{Remark}$  The Killing form of $\fH$ is the restriction to $\fH$ of the Killing form of $\fG$ since $\fH$ is an ideal of $\fG$. 

\section{The Clifford Lie algebras $\fG_{r,s}$}

\noindent  Let r,s,n be integers such that $r \geq 0, s > 0$ and $r+s = n$.  The case $s = 0$ was treated in [E].  Let $\{e_{1}, ... , e_{n} \}$ be the canonical basis for $V = F^{n}$, and let Q be the symmetric bilinear form on $F^{n}$ such that $Q(e_{i} , e_{j}) = 0$ if $i \neq j, Q(e_{i} , e_{i}) = 1$ if $1 \leq i \leq r$ and $Q(e_{i} , e_{i}) = -1$ if $r+1 \leq i \leq r+s$.  Let $C\ell(r,s)$ denote the Clifford algebra determined by $F^{n}$ and Q, and let $\fG_{r,s} = \{\varphi \in C\ell(r,s) : c(\varphi) = - \varphi \}$.  $C\ell(r,s)$ becomes a Lie algebra by defining $[x,y] = xy - yx$ for all $x,y \in C\ell(r,s)$. 
\newline

\noindent  $\mathbf{Canonical~automorphism~\beta~ of~C\ell(r,s)}$

\begin{proposition}Let $\beta : F^{n} \rightarrow F^{n}$ be the linear map such that $\beta(e_{i}) = e_{i}$ for $1 \leq i \leq r$ and $\beta(e_{r+j}) = - e_{r+j}$ for $1 \leq j \leq s$.  Then $\beta$ extends to an algebra automorphism of $C\ell(r,s)$.
\end{proposition}

\begin{proof}  It suffices to show that $\beta$ extends to an algebra $\mathit{homomorphism}$ of $C\ell(r,s)$.   In this case it follows that $\beta \circ \beta = Id$ on V, and hence on $C\ell(r,s)$ by the uniqueness part of Proposition 1.1.  By Proposition 1.1 we need only show that $\beta(v) \cdot \beta(v) = - Q(v,v)$ for all $v \in F^{n}$.
\newline

\noindent Let $v \in F^{n}$ be given and write $v = \sum_{i=1}^{r} a_{i} e_{i} + \sum_{j=1}^{s} b_{j} e_{r+j}$ for suitable elements $a_{i},b_{j}$ of F.  Then $\beta(v) = \sum_{i=1}^{r} a_{i} e_{i} - \sum_{j=1}^{s} b_{j} e_{r+j}$ and $\beta(v) \cdot \beta(v) = \sum_{i=1}^{r} a_{i}^{2} e_{i}^{2} + \sum_{1=m<n}^{r} a_{m}a_{n} \{e_{m} \cdot e_{n} + e_{n} \cdot e_{m}\} - \sum_{i=1}^{r} \sum_{j=1}^{s} a_{i} b_{j} \{e_{i} \cdot e_{r+j} + e_{r+j} \cdot e_{i} \} + 
\newline \sum_{1=m < n}^{s} b_{m} b_{n} \{e_{r+m} \cdot e_{r+n} + e_{r+n} \cdot e_{r+m} \} + \sum_{j=1}^{s} b_{j}^{2} e_{r+j}^{2} = \sum_{i=1}^{r} a_{i}^{2} e_{i}^{2} + \sum_{j=1}^{s} b_{j}^{2} e_{r+j}^{2} = - \sum_{i=1}^{r} a_{i}^{2} + \sum_{j=1}^{s} b_{j}^{2}$. 
\newline

\noindent Since $\{e_{1}, ... , e_{n} \}$ is a Q-orthogonal basis of $V = F^{n}$ it follows that $Q(v,v) = \sum_{i=1}^{r} a_{i}^{2}~ Q(e_{i}, e_{i}) + \sum_{j=1}^{s} b_{j}^{2}~ Q(e_{r+j}, e_{r+j}) = \sum_{i=1}^{r} a_{i}^{2} - \sum_{j=1}^{s} b_{j}^{2} = - \beta(v) \cdot \beta(v)$. 
\end{proof}

\begin{corollary}  The anti-automorphism $c : C\ell(r,s) \rightarrow C\ell(r,s)$ and the automorphisms $\alpha, \beta : C\ell(r,s) \rightarrow C\ell(r,s)$ all commute.
\end{corollary}

\begin{proof}  We show that $\alpha$ and $\beta$ commute.  The proofs of the other assertions are similar and are omitted.  Let $A = \{\xi \in C\ell(r,s) : \alpha \beta(\xi) =  \beta \alpha(\xi)\}$.   The set A is a subalgebra of $C\ell(r,s)$ since $\alpha$ and $\beta$ are automorphisms.  It is easy to check that $e_{i} \in A$ for $1 \leq i \leq r+s$.  The assertion now follows immediately since $\{e_{1}, ... , e_{r+s} \}$ generates $C\ell(r,s)$ as an algebra.
\end{proof}

\begin{corollary}  The automorphism $\beta$ leaves $\fG_{r,s}$ invariant.
\end{corollary}

\begin{proof}  $\beta$ leaves the eigenspaces of c invariant since $\beta$ and c commute.  $\fG_{r,s}$ is by definition the $-1$ eigenspace of c.
\end{proof}

\noindent $\mathbf{Cartan~decomposition~of~C\ell(r,s)}$

\noindent  Let $\fK_{r,s}$ be the $+1$ eigenspace of $\beta$ restricted to $\fG_{r,s}$, and let $\fP_{r,s}$ be the $- 1$ eigenspace of $\beta$ restricted to $\fG_{r,s}$.  Then $\fG_{r,s} = \fK_{r,s} \oplus \fP_{r,s}$.  This direct sum is called the $\mathit{Cartan~decomposition}$ of $\fG_{r,s}$.  It follows from the linear independence of $\{e_{I} : I \in \fI \}$ that $\fG_{r,s} = F-span \{e_{I} : e_{I} \in \fG_{r,s} \}, \fK_{r,s} = F-span \{e_{I} : e_{I} \in \fK_{r,s} \}$ and $\fP_{r,s} = F-span \{e_{I} : e_{I} \in \fP_{r,s} \}$.
\newline  

\noindent Note that $\beta$ is a Lie algebra automorphism of $\fG_{r,s}$ since $\beta$ is an algebra automorphism of $\fG_{r,s}$.  As an immediate consequence we obtain

\begin{proposition} The following bracket relations hold :

\hspace{.5in} $1) [\fK_{r,s}, \fK_{r,s}] \subset \fK_{r,s}$ 

\hspace{.5in} $ 2) [\fK_{r,s}, \fP_{r,s}] \subset \fP_{r,s}$ 

\hspace{.5in} $3)  [\fP_{r,s}, \fP_{r,s}] \subset \fK_{r,s}$. 
\end{proposition}

\noindent  The next result characterizes $\fK_{r,s}$ and $\fP_{r,s}$ in a way that will be useful later for computing their dimensions.  Recall that for a multi-index $I = (i_{1}, i_{2}, ... , i_{k})$, where $i_{1} < i_{2} < ... < i_{k}$, we let $e_{I} = e_{i_{1}} \cdot e_{i_{2}} \cdot ... \cdot e_{i_{k}}$.  For $e_{I} \in \fG_{r,s}$ it is easy to see by induction on $|I|$ that $e_{I}^{2} = 1$ or $-1$ for all $I \in \fI$. 

\begin{proposition}  The following statements hold :

\hspace{.5in}  1)  $\fK_{r,s} = \br-$span$\{e_{I} : e_{I} \in \fG_{r,s}~\rm{and}~ e_{I}^{2} = -1\}$

\hspace{.5in}  2)  $\fP_{r,s} = \br-$span$\{e_{I} : e_{I} \in \fG_{r,s}~\rm{and}~  e_{I}^{2} = 1\}$
\end{proposition}

\begin{lemma}  Let $I = (i_{1}, ... , i_{k})$ be a multi-index.  Then $e_{I}^{2} = (-1)^{\frac{k(k-1)}{2}} e_{i_{1}}^{2} e_{i_{2}}^{2} ... e_{i_{k}}^{2}$.
\end{lemma}

\noindent  The proof follows routinely by induction on $|I|$, and we omit the details.
\newline

\noindent  For the next result we introduce some notation.  Given a multi-index $I = (i_{1}, ... , i_{k})$ let $I^{-} = \{i_{r} : 1 \leq r \leq k~\rm{and}~e_{i_{r}}^{2} = -1 \}$ and $I^{+} = \{i_{r} : 1 \leq r \leq k~\rm{and}~e_{i_{r}}^{2} = 1 \}$.  Equivalently, $I^{+} = I \cap \{1,2, ... , r \}$ and $I^{-} = I \cap \{r+1,r+2, ... , r+s \}$

\begin{lemma}  Let $I = (i_{1}, ... , i_{k})$ be a multi-index such that $e_{I} \in \fG_{r,s}$.  Then $e_{I}^{2} = 1 \Leftrightarrow |I^{+}|$ is an odd integer. 
\end{lemma}

\begin{proof}  By Proposition 1.5 $e_{I} \in \fG_{r,s} \Leftrightarrow |I| \equiv 1$ or $2~(mod~4)$.  If $ k = |I| = 4\alpha +1$ for some integer $\alpha \geq 0$, then $e_{I}^{2} = e_{i_{1}}^{2} e_{i_{2}}^{2} ... e_{i_{k}}^{2}$ by Lemma 2.6.  Hence $e_{I}^{2} = 1  \Leftrightarrow |I^{-}|$ is even, but this occurs $\Leftrightarrow |I^{+}| = 4\alpha +1 - |I^{-}|$ is odd.  If $ k = |I| = 4\alpha + 2$ for some integer $\alpha \geq 0$, then $e_{I}^{2} = - e_{i_{1}}^{2} e_{i_{2}}^{2} ... e_{i_{k}}^{2}$ by Lemma 2.6.  Hence $e_{I}^{2} = 1 \Leftrightarrow |I^{-}|$ is odd, but this occurs $\Leftrightarrow |I^{+}| = 4\alpha +2 - |I^{-}|$ is odd.
\end{proof}

\begin{lemma}  $\beta(e_{I}) = - e_{I} \Leftrightarrow |I^{+}|$ is odd.
\end{lemma}

\begin{proof}  If $e_{I} = e_{i_{1}} \cdot e_{i_{2}} \cdot ... \cdot e_{i_{k}}$, then $\beta(e_{I}) = \beta(e_{i_{1}}) \cdot  \beta(e_{i_{2}}) \cdot ... \cdot  \beta(e_{i_{k}}) = (- 1)^{|I^{+}|} e_{I}$ since $\beta$ is an automorphism of $\fG_{r,s}$. 
\end{proof}

\noindent  We are now ready to prove the Proposition.  Let $\fK^{\prime}_{r,s} = \br-$ span $\{e_{I} \in \fG_{r,s} : e_{I}^{2} = -1 \}$ and let $\fP^{\prime}_{r,s} = \br-$ span $\{e_{I} \in \fG_{r,s} : e_{I}^{2} = 1 \}$.  By Proposition 1.5 it is evident that  $\fG_{r,s} = \fK^{\prime}_{r,s} \oplus \fP^{\prime}_{r,s}$ and from the definitions it follows that $\fK_{r,s} \cap \fP_{r,s} = \{0 \}$.  Hence it suffices to show that $\fK^{\prime}_{r,s} \subseteq \fK_{r,s}$ and $\fP^{\prime}_{r,s} \subseteq \fP_{r,s}$.  Suppose that $e_{I}^{2} = -1$.  Then $|I^{+}|$ is even by Lemma 2.7, and hence $\beta(e_{I}) = e_{I}$ by Lemma 2.8 since $\beta(e_{I}) = \pm e_{I}$ for all multi-indices I.  Hence $e_{I} \in \fK_{r,s}$ and it follows that $\fK^{\prime}_{r,s} \subseteq \fK_{r,s}$.  Similarly, if $e_{I}^{2} = 1$, then $e_{I} \in \fP_{r,s}$ by Lemma 2.8  and $\fP^{\prime}_{r,s} \subseteq \fP_{r,s}$.
\newline 

\noindent  Let $r,s$ be integers such that $r \geq 0$ and $s > 0$.  Our next goal is to determine the dimensions of $\fG_{r,s}, \fK_{r,s}$ and $\fP_{r,s}$ as functions of r and s.  
\newline

\noindent $\mathbf{Dimension~ of~ \fG_{r,s}}$

\begin{proposition}  The dimension of $\fG_{r,s}$ is $2^{r+s-1} - 2^{\frac{r+s-1}{2}}~cos~[\frac{(r+s+1) \pi}{4}]$.
\end{proposition}

\begin{lemma}  Let a be an integer with $0 \leq a \leq 3$.  Then 
\newline

$\sum_{k \geq 0} \left(\begin{array}{ccc}n \\ a+4k\\ \end{array} \right) = \frac{1}{4}\{2^{n} + i^{-a}(1+i)^{n} + i^{-3a}(1-i)^{n} \}$.   
\end{lemma}

\begin{proof}  This is the special case $r=4$ of the formula $\sum_{k \geq 0} \left(\begin{array}{ccc}n \\ a+rk\\ \end{array} \right) = \newline  \frac{1}{r} \sum_{j=0}^{r-1} \omega^{-ja}(1+\omega^{j})^{n}$, where $\omega = e^{\frac {2 \pi i}{r}}$.  See [BCK].
\end{proof}

\begin{corollary} The following identities hold :

\noindent 1) $\sum_{k \geq 0} \left(\begin{array}{ccc}n \\ 4k\\ \end{array} \right) = \frac{1}{4}\{2^{n} +(1+i)^{n} +  (1-i)^{n} \} = 2^{n-2} + 2^{\frac {n-2}{2}} cos(\frac{n \pi}{4})$.

\noindent 2) $\sum_{k \geq 0} \left(\begin{array}{ccc}n \\ 4k+1\\ \end{array} \right) = \frac{1}{4}\{2^{n} -i (1+i)^{n} + i (1-i)^{n} \} = 2^{n-2} + 2^{\frac {n-2}{2}} sin(\frac{n \pi}{4})$.
 
\noindent 3) $\sum_{k \geq 0} \left(\begin{array}{ccc}n \\ 4k+2\\ \end{array} \right) = \frac{1}{4}\{2^{n} - (1+i)^{n} - (1-i)^{n} \} = 2^{n-2} - 2^{\frac {n-2}{2}} cos(\frac{n \pi}{4})$.

\noindent 4) $\sum_{k \geq 0} \left(\begin{array}{ccc}n \\ 4k+3\\ \end{array} \right) = \frac{1}{4}\{2^{n} + i (1+i)^{n} - i (1-i)^{n} \} = 2^{n-2} - 2^{\frac {n-2}{2}} sin(\frac{n \pi}{4})$
\end{corollary}

\begin{proof}  In each of the cases 1) through 4) the first equality follows by direct substitution into the lemma above while the second equality follows directly from the facts that $1 + i = 2^{\frac{1}{2}} e^{\frac{i \pi}{4}}$ and  $1 - i = 2^{\frac{1}{2}} e^{\frac{ -i \pi}{4}}$
\end{proof}

\noindent  We now complete the proof of the Proposition.   We recall from Proposition 1.5 that $\fG = span \{e_{I} : e_{I} \in \fG\}$ and $e_{I} \in \fG \Leftrightarrow |I| \equiv 1$ or $2~(mod~4)$. Using the corollary above we obtain $dim~\fG = \sum_{k \geq 0} \left(\begin{array}{ccc}r+s \\ 4k+1\\ \end{array} \right) + \sum_{k \geq 0} \left(\begin{array}{ccc}r+s \\ 4k+2\\ \end{array} \right) = \newline 2^{r+s-1} - 2^{\frac{r+s-2}{2}}[cos(\frac{(r+s) \pi}{4}) - sin(\frac{(r+s) \pi}{4})] = 2^{r+s-1} - 2^{\frac{r+s-2}{2}} ~2^{\frac{1}{2}}~ cos(\frac{(r+s+1) \pi}{4}) = 2^{r+s-1} - 2^{\frac{r+s-1}{2}} ~ cos(\frac{(r+s+1) \pi}{4})$.
\newline

\noindent $\mathbf{Dimensions~ of~ \fP_{r,s}~and~\fK_{r,s}}$

\begin{proposition}  We have the following equalities

	1) If $r \geq 3$ and $s \geq 3$, then $dim~\fP_{r,s} = 2^{r+s-2} + 2^{\frac{r+s-1}{2}} sin~\frac{(r+1) \pi}{4}~sin~\frac{s \pi}{4}$.
	
	2)  If $r \geq 3$, then 
	
\hspace{.5in} a)  dim $\fP_{r,1} = 2^{r-1} + 2^{\frac{r-1}{2}}~sin(\frac{(r+1) \pi}{4})$.
		
\hspace{.5in}	b)  dim $\fP_{r,2} = 2^{r} + 2^{\frac{r+1}{2}}~sin(\frac{(r+1) \pi}{4})$

	3)  If $s \geq 3$, then
	
\hspace{.5in} a)  dim $\fP_{1,s} = 2^{s-1} + 2^{\frac{s}{2}}~sin(\frac{s \pi}{4})$.
		
\hspace{.5in}	b)  dim $\fP_{2,s} = 2^{s} + 2^{\frac{s}{2}}~sin(\frac{s \pi}{4})$

	4)  dim $\fP_{1,1} = 2$   \hspace{.2in} dim $\fP_{1,2} = 4$  \hspace{.2in} dim $\fP_{2,1} = 3$ \hspace{.2in} dim $\fP_{2,2} = 6$
\end{proposition}

\begin{proposition}  We have the following equalities

	1) If $r \geq 3$ and $s \geq 3$, then $dim~\fK_{r,s} = 2^{r+s-2} - 2^{\frac{r+s-1}{2}} cos~\frac{(r+1) \pi}{4}~cos~\frac{s \pi}{4} =  2^{r+s-2} - 2^{\frac{r+s-3}{2}} \{cos \frac{(r+s+1)\pi}{4} + cos \frac{(r-s+1)\pi}{4}\}$.
	
	2)  If $r \geq 3$, then 
	
\hspace{.5in} a)  dim $\fK_{r,1} = 2^{r-1} - 2^{\frac{r-1}{2}}~cos(\frac{(r+1) \pi}{4})$.
		
\hspace{.5in}	b)  dim $\fK_{r,2} = 2^{r}$

	3)  If $s \geq 3$, then
	
\hspace{.5in} a)  dim $\fK_{1,s} = 2^{s-1}$.
		
\hspace{.5in}	b)  dim $\fK_{2,s} = 2^{s} + 2^{\frac{s}{2}}~cos(\frac{s \pi}{4})$

	4)  dim $\fK_{1,1} = 1$   \hspace{.2in} dim $\fK_{1,2} = 2$  \hspace{.2in} dim $\fK_{2,1} = 3$ \hspace{.2in} dim $\fK_{2,2} = 4$
\end{proposition}

\noindent  $\mathit{Proof~of~1)}$We now begin the proof of  Proposition 2.12.  By Propositions 1.5 and 2.5 and Lemma 2.8 we see that $\fP_{r,s} = \br-\rm{span}~\{e_{I} : |I| \equiv 1~\rm{or}~2~mod(4) \rm{and}~|I^{+}|~\rm{is~odd} \}$.   We divide the proof of  part 1) of Proposition 2.12  into four cases, some of which do not occur in the proofs of parts 2) and 3).  Let $r \geq 3$ and $s \geq 3$. 
\newline

\noindent Case 1A $|I| = 4\alpha+1, |I^{+}| = 4 \beta +1$ 

\noindent Case 1B $|I| = 4 \alpha + 1, |I^{+}| = 4 \beta + 3$ 

\noindent Case 2A $|I| = 4\alpha+2, |I^{+}| = 4 \beta +1$ 

\noindent  Case 2B $|I| = 4 \alpha + 2, |I^{+}| = 4 \beta + 3$.
\newline

\noindent  Case 1A :  Let $\beta \geq 0$ be an integer such that $|I^{+}| = 4 \beta +1$.  Then $|I^{-}| = 4 \alpha +1 - (4 \beta + 1) = 4(\alpha - \beta)$.  Hence the number of multi-indices I with $|I| = 4 \alpha + 1$ and $|I^{+}| = 4 \beta + 1$ equals $\left(\begin{array}{ccc}s \\ 4 \beta + 1\\ \end{array} \right) \{\left(\begin{array}{ccc}r \\ 0\\ \end{array} \right) + \left(\begin{array}{ccc}r \\ 4\\ \end{array} \right) + \left(\begin{array}{ccc}r \\ 8\\ \end{array} \right) + ...  \} =  \left(\begin{array}{ccc}s \\ 4 \beta + 1\\ \end{array} \right) (2^{r-2} + 2^{\frac{r-2}{2}} cos~\frac{r \pi}{4})$ by 1) of Corollary 2.11.  Using 2) of Corollary 2.11 and summing over the integers $\beta$ with $4 \beta + 1 \leq s$ we obtain 
\newline

\noindent (1A) $\{2^{r-2} + 2^{\frac{r-2}{2}} cos(\frac{r \pi}{4}) \} \{2^{s-2} + 2^{\frac{s-2}{2}} sin(\frac{s \pi}{4}) \}$ multi-indices I satisfying Case 1A.
\newline

\noindent  Case 1B :  Let $\beta \geq 0$ be an integer such that $|I^{+}| = 4 \beta +3$.  Then $|I^{-}| = 4 \alpha +1 - (4 \beta + 3) = 4(\alpha - \beta -1) + 2$.  Hence the number of multi-indices I with $|I| = 4 \alpha + 1$ and $|I^{+}| = 4 \beta + 3$ equals $\left(\begin{array}{ccc}s \\ 4 \beta + 3\\ \end{array} \right) \{\left(\begin{array}{ccc}r \\ 2\\ \end{array} \right) + \left(\begin{array}{ccc}r \\ 6\\ \end{array} \right) + \left(\begin{array}{ccc}r \\ 10\\ \end{array} \right) + ...  \} =  \left(\begin{array}{ccc}s \\ 4 \beta + 3\\ \end{array} \right) (2^{r-2} - 2^{\frac{r-2}{2}}~cos~\frac{r \pi}{4})$ by 3) of Corollary 2.11.  Using 4) of Corollary 2.11 and summing over the integers $\beta$ with $4 \beta + 3 \leq s$ we obtain
\newline

\noindent (1B) $\{2^{r-2} - 2^{\frac{r-2}{2}} cos(\frac{r \pi}{4}) \} \{2^{s-2} - 2^{\frac{s-2}{2}} sin(\frac{s \pi}{4}) \}$ multi-indices I satisfying Case 1B.
\newline

\noindent  Summing the results of 1A and 1B we obtain
\newline

\noindent (1) $2^{r+s-3} + 2^{\frac{r+s-2}{2}}~sin(\frac{s \pi}{4})~cos(\frac{r \pi}{4})$ multi-indices I satisfying Case 1.
\newline

\noindent \noindent  Case 2A :  Let $\beta \geq 0$ be an integer such that $|I^{+}| = 4 \beta +1$.  Then $|I^{-}| = 4 \alpha +2 - (4 \beta + 1) = 4(\alpha - \beta) + 1$.  Hence the number of multi-indices I with $|I| = 4 \alpha + 2$ and $|I^{+}| = 4 \beta + 1$ equals $\left(\begin{array}{ccc}s \\ 4 \beta + 1\\ \end{array} \right)  \{\left(\begin{array}{ccc}r \\ 1\\ \end{array} \right) + \left(\begin{array}{ccc}r \\ 5\\ \end{array} \right) + \left(\begin{array}{ccc}r \\ 9\\ \end{array} \right) + ...  \} =  \left(\begin{array}{ccc}s \\ 4 \beta + 1\\ \end{array} \right) (2^{r-2} + 2^{\frac{r-2}{2}}~sin~\frac{r \pi}{4})$ by 2) of Corollary 2.11.  Using 2) of  Corollary 2.11 again  and summing over the integers $\beta$ with $4 \beta + 1 \leq s$ we obtain
\newline 

\noindent (2A) $\{2^{r-2} + 2^{\frac{r-2}{2}} sin(\frac{r \pi}{4}) \} \{2^{s-2} + 2^{\frac{s-2}{2}} sin(\frac{s \pi}{4}) \}$ multi-indices I satisfying Case 2A.
\newline

\noindent  Case 2B :  Let $\beta \geq 0$ be an integer such that $|I^{+}| = 4 \beta +3$.  Then $|I^{-}| = 4 \alpha +2 - (4 \beta + 3) = 4(\alpha - \beta -1) + 3$.  Hence the number of multi-indices I with $|I| = 4 \alpha + 1$ and $|I^{+}| = 4 \beta + 3$ equals $\left(\begin{array}{ccc}s \\ 4 \beta + 3\\ \end{array} \right) \{\left(\begin{array}{ccc}r \\ 3\\ \end{array} \right) + \left(\begin{array}{ccc}r \\ 7\\ \end{array} \right) + \left(\begin{array}{ccc}r \\ 11\\ \end{array} \right) + ...  \} =  \left(\begin{array}{ccc}s \\ 4 \beta + 3\\ \end{array} \right) (2^{r-2} - 2^{\frac{r-2}{2}}~sin~\frac{r \pi}{4})$ by 4) of Corollary 2.11.  Using 4) of Corollary 2.11 again and summing over the integers $\beta$ with $4 \beta + 3 \leq s$ we obtain
\newline

\noindent (2B) $\{2^{r-2} - 2^{\frac{r-2}{2}} sin(\frac{r \pi}{4}) \} \{2^{s-2} - 2^{\frac{s-2}{2}} sin(\frac{s \pi}{4}) \}$ multi-indices I satisfying Case 2B.
\newline

\noindent  Summing the results of 2A and 2B we obtain
\newline

\noindent (2) $2^{r+s-3} + 2^{\frac{r+s-2}{2}}~sin(\frac{s \pi}{4})~sin(\frac{r \pi}{4})$ multi-indices I satisfying Case 2.
\newline

\noindent  Finally the sum of the number of multi-indices I in Case 1 and Case 2 equals $\{2^{r+s-3} +2^{\frac{r+s-2}{2}}~sin~(\frac{s \pi}{4})~cos~(\frac{r \pi}{4})\} + \{2^{r+s-3} +2^{\frac{r+s-2}{2}}~sin~(\frac{s \pi}{4})~sin~(\frac{r \pi}{4}) \} = \newline 2^{r+s-2} + 2^{\frac{r+s-2}{2}} ~sin~(\frac{s \pi}{4}) \{cos(\frac{r \pi}{4}) + sin(\frac{r \pi}{4}) \} = 2^{r+s-2} + 2^{\frac{r+s-1}{2}} ~sin~(\frac{s \pi}{4}) sin(\frac{(r+1) \pi}{4})$ since $sin(\frac{(r+1) \pi}{4}) = 2^{-\frac{1}{2}}~(sin~(\frac{r \pi}{4}) + cos~(\frac{r \pi}{4})$.  This completes the proof of  part 1) of Proposition 2.12.
\newline

\noindent  $\mathit{Proof~ of~ 2)}$.   If $s =1$ or $s = 2$, then $|I^{+}| \leq s \leq 2$.  Hence Cases 1B and 2B do not apply in this case.
\newline

\noindent  Let $s = 1$ and $r \geq 3$.  If we repeat the argument above with $s=1$, then Case 1A yields $ 2^{r-2} + 2^{\frac{r-2}{2}}~cos(\frac{r \pi}{4})$ multi-indices I.  Similarly, Case 2A yields  $ 2^{r-2} + 2^{\frac{r-2}{2}}~sin(\frac{r \pi}{4})$ multi-indices I.  The sum of the multi-indices in Cases 1A and 2A equals $2^{r-1} + 2^{\frac{r-2}{2}} [cos(\frac{r \pi}{4}) + sin(\frac{r \pi}{4})] =  2^{r-1} + 2^{\frac{r-1}{2}}  sin(\frac{(r+1) \pi}{4}) $.  This proves 2a).
\newline

\noindent Let $s = 2$ and $r \geq 3$.  Then Case 1A yields $\{2^{r-2} + 2^{\frac{r-2}{2}}~cos (\frac{r \pi}{4}) \} \{1+ ~sin(\frac{\pi}{2}) \} = 2^{r-1} + 2^{\frac{r}{2}}~cos(\frac{r \pi}{4})$ multi-indices I.   Case 2A yields $\{2^{r-2} + 2^{\frac{r-2}{2}}~sin (\frac{r \pi}{4}) \} \{1+ ~sin(\frac{\pi}{2}) \} = 2^{r-1} + 2^{\frac{r}{2}}~sin(\frac{r \pi}{4})$ multi-indices I.  The sum of the multi-indices in Cases 1A and 2A equals $2^{r} + 2^{\frac{r}{2}} [cos(\frac{r \pi}{4}) + sin(\frac{r \pi}{4})] = 2^{r} + 2^{\frac{r+1}{2}}~sin(\frac{(r+1) \pi}{4})$.  This proves 2b).
\newline

\noindent $\mathit{Proof~ of~ 3)}$.  Let $r=1$ and $s \geq 3$.  Then Cases 1B and 2B do not arise.  In case 1A we have $\sum_{\beta \geq 0} \left(\begin{array}{ccc}s \\ 4 \beta + 1\\ \end{array} \right) = 2^{s-2} + 2^{\frac{s-2}{2}}~sin(\frac{s \pi}{4})$ multi-indices I.  In Case 2A we also have  $\sum_{\beta \geq 0} \left(\begin{array}{ccc}s \\ 4 \beta + 1\\ \end{array} \right) = 2^{s-2} + 2^{\frac{s-2}{2}}~sin(\frac{s \pi}{4})$ multi-indices I.  The sum of the multi-indices from Cases 1A and 2A equals $2^{s-1} + 2^{\frac{s}{2}}~sin(\frac{s \pi}{4})$.  This proves 3a).
\newline

\noindent Now let $r=2$ and $s \geq 3$.  Then Case 2B does not occur.  In case 1A we have $\sum_{\beta \geq 0} \left(\begin{array}{ccc}s \\ 4 \beta + 1\\ \end{array} \right) = 2^{s-2} + 2^{\frac{s-2}{2}}~sin(\frac{s \pi}{4})$ multi-indices I.  In Case 1B we have $\sum_{\beta \geq 0} \left(\begin{array}{ccc}s \\ 4 \beta + 3\\ \end{array} \right) = 2^{s-2} - 2^{\frac{s-2}{2}}~sin(\frac{s \pi}{4})$ multi-indices I.  In Case 2A we have $2~(\sum_{\beta \geq 0} \left(\begin{array}{ccc}s \\ 4 \beta + 1\\ \end{array} \right)) = 2~(2^{s-2} + 2^{\frac{s-2}{2}}~sin(\frac{s \pi}{4})) = 2^{s-1} + 2^{\frac{s}{2}}~sin(\frac{s \pi}{4})$ multi-indices I.  The sum of the multi-indices from Cases 1A, 1B and 2A equals $2^{s} + 2^{\frac{s}{2}}~sin(\frac{s \pi}{4})$.  This proves 3b).
\newline

\noindent  $\mathit{Proof~ of~ 4)}$.  These are easy computations directly from the definition of $\fP_{r,s}$ in the cases $r \leq 2$ and $s \leq 2$.
\newline

\noindent  We now prove Proposition 2.13, but only in Case 1).  The other cases follow routinely from trigonometric identities.
\newline

\noindent Let $r \geq 3$ and $s \geq 3$.  Since $\fG_{r,s} = \fK_{r,s} \oplus \fP_{r,s}$ it follows from Propositions 2.9 and 2.12 that dim $\fK_{r,s} = dim~\fG_{r,s} - dim~\fP_{r,s} = \{ 2^{r+s-1} - 2^{\frac{r+s-1}{2}} cos(\frac{(r+s+1) \pi}{4}) \} - \{2^{r+s-2} + 2^{\frac{r+s-1}{2}} sin(\frac{s \pi}{4}) sin(\frac{(r+1) \pi}{4}) \} = 2^{r+s-2} - 2^{\frac{r+s-1}{2}} \{cos(\frac{(r+1) \pi}{4}) cos(\frac{s \pi}{4}) - sin(\frac{(r+1) \pi}{4}) sin(\frac{s \pi}{4}) +  sin(\frac{(r+1) \pi}{4}) sin(\frac{s \pi}{4})\}  = 2^{r+s-2} - 2^{\frac{r+s-1}{2}} cos(\frac{(r+1) \pi}{4} cos(\frac{s \pi}{4})$.
\newline

\section{Commutation relations}

\noindent  Let $I,J \in \fI$ be given.  It is easy to see by induction on the minimum of $|I|,|J|$ that $e_{I} e_{J} = \pm e_{J} e_{I}$.  The next two results determine the sign $\pm$ for each pair I,J.  These relations are valid in an arbitrary Clifford algebra $C\ell(V,Q)$, where V is a finite dimensional vector space over a field F with characteristic $\neq 2$.

\begin{proposition}  Let $I,J \in \fI$ be given with I,J disjoint.  Then $e_{I} e_{J} = (-1)^{|I||J|} e_{J} e_{I}$.
\end{proposition}

\begin{proposition}  Let $I,J \in \fI$ be given with $I \cap J$ nonempty. 

   1)  If $|I \cap J|$ is even, then $e_{I} e_{J} = - e_{J} e_{I} \Leftrightarrow |I|$ and $|J|$ are both odd.
   
   2)  If $|I \cap J|$ is odd, then $e_{I} e_{J} =  e_{J} e_{I} \Leftrightarrow |I|$ and $|J|$ are both odd.
\end{proposition} 

\noindent  We omit the proof of  Proposition 3.1, which follows by induction on $|I| + |J|$. We prove Proposition 3.2.  The elements $e_{I}, e_{J}$ commute (anti-commute) $\Leftrightarrow \pm e_{I}, \pm e_{J}$ commute (anti-commute) for each of the four choices of signs.  Hence, without loss of generality, we may assume that $e_{I} = e_{I \cap J}~ e_{I_{1}}$ and $e_{J} = e_{I \cap J}~ e_{J_{1}}$, where $I = (I \cap J) \cup~I_{1}$ (disjoint union) and $J = (I \cap J) \cup~J_{1}$ (disjoint union).  Note that any two of the sets $I_{1},J_{1}, I \cap J$ are disjoint.
\newline

\noindent 1)  Suppose first that $|I \cap J|$ is even.  Using the previous result we compute $e_{I} e_{J} = e_{I \cap J}~ e_{I_{1}} e_{I \cap J}~ e_{J_{1}} = (-1)^{|I \cap J| |I_{1}|} (e_{I \cap J})^{2} e_{I_{1}} e_{J_{1}} = (e_{I \cap J})^{2} e_{I_{1}} e_{J_{1}} = \newline  (-1)^{|I_{1}| |J_{1}|} (e_{I \cap J})^{2} e_{J_{1}} e_{I_{1}}$.  We have proved

	a)  $e_{I} e_{J} = (-1)^{|I_{1}| |J_{1}|} (e_{I \cap J})^{2} e_{J_{1}} e_{I_{1}}$.
	
\noindent  Similarly we compute $e_{J} e_{I} = (-1)^{|I \cap J| |J_{1}|} (e_{I \cap J})^{2} e_{J_{1}} e_{I_{1}} = (e_{I \cap J})^{2} e_{J_{1}} e_{I_{1}}$.  We have proved

	b)   $e_{J} e_{I} = (e_{I \cap J})^{2} e_{J_{1}} e_{I_{1}}$.
\noindent From a) and b) it follows that $e_{I} e_{J} = - e_{J} e_{I} \Leftrightarrow (-1)^{|I_{1}| |J_{1}|} = -1 \Leftrightarrow |I_{1}|$ and $|J_{1}|$ are both odd $\Leftrightarrow |I|$ and $|J|$ are both odd since $|I \cap J|$ is even.  This proves part 1).
\newline

\noindent 2)  Suppose next that $|I \cap J|$ is odd.  As above we compute \newline $e_{I} e_{J} = (-1)^{|I \cap J| |I_{1}|} (e_{I \cap J})^{2} e_{I_{1}} e_{J_{1}} = (-1)^{|I_{1}|} (e_{I \cap J})^{2} e_{I_{1}} e_{J_{1}} = \newline (-1)^{(|I_{1}| + |I_{1}| |J_{1}|)} (e_{I \cap J})^{2} e_{J_{1}} e_{I_{1}}$.  We have proved

	a)  $e_{I} e_{J} = (-1)^{(|I_{1}| + |I_{1}| |J_{1}|)} (e_{I \cap J})^{2} e_{J_{1}} e_{I_{1}}$.

\noindent Similarly we obtain 

	b)  $e_{J} e_{I} = (-1)^{|J_{1}|} (e_{I \cap J})^{2} e_{J_{1}} e_{I_{1}}$.
	
\noindent  From a) and b) we see that $e_{I} e_{J} = e_{J} e_{I} \Leftrightarrow$

	$(^{*})\hspace{.2in} (-1)^{(|I_{1}| + |I_{1}| |J_{1}|)} = (-1)^{|J_{1}|}$.
	
\noindent  If $|J_{1}|$ is odd, then $(^{*})$ becomes $-1 = 1$, and we conclude that $|J_{1}|$ is even.  The condition $(^{*})$ now becomes $1 = (-1)^{|I_{1}| (1 + |J_{1}|)}$, which implies that $|I_{1}|$ is also even.  Hence $|I|$ and $|J|$ are both odd since $|I \cap J|$ is odd.  Conversely, if $|I|$ and $|J|$ are both odd, then $|I_{1}|$ and $|J_{1}|$ are both even and $(^{*})$ holds.  This proves part 2).

\section{Center of $\fK_{r,s}$}

\noindent  One of the main tools in determining $\fG_{r,s}$ when $F = \br$, will be a description of $\fZ(\fK_{r,s})$, the center of $\fK_{r,s}$.  We achieve that in this section, where we still assume that F is an arbitrary field of characteristic $\neq 2$. The main results  are the following :

\begin{proposition}  $\fZ(\fK_{r,s}) = $F-span $\{e_{I} : e_{I} \in \fZ(\fK_{r,s}) \}$
\end{proposition}

\begin{proposition}  Let $I \in \fI$ be given.  Then $e_{I} \in \fZ(\fK_{r,s}) \Leftrightarrow$ one of the following occurs :

	a)  $e_{I} = \omega = e_{1} e_{2} ... e_{r+s}, \hspace{.2in}  r+s \equiv 1~mod(4)$ and s is even.
	
	b)  $I = \{r+1, r+2, ... ,r+s \}$ and  $s \equiv 2~mod(4)$
	
	c)   $I  = \{1, 2, ... ,r \}$ and  $r \equiv 1~mod(4)$
\end{proposition}

\noindent $\mathit{Proof~of~Proposition~4.1}$  This is an immediate consequence of the following result :

\begin{lemma}  Let $\fI_{\fK} = \{I \in \fI : e_{I} \in \fK_{r,s} \}$.  Let $\xi \in \fZ(\fK_{r,s})$ be given and write $\xi = \sum_{I \in \fI_{\fK}} \xi_{I} e_{I}$, where $\xi_{I} \in F$ for all I.  If $\xi_{I} \neq 0$, then $e_{I} \in \fZ(\fK_{r,s})$.
\end{lemma}

\begin{proof}  Recall that $\fK_{r,s} = F-span \{e_{I} : I \in \fI_{k} \}$.  Let $e_{K} \in \fK_{r,s}$.  Then $0 = [e_{K}, \xi ] = e_{K} \xi - \xi e_{K} = \sum_{I \in \fI_{\fK}} \xi_{I} (e_{K}e_{I} - e_{I} e_{K}) = \sum_{I \in \fI_{\fK}} \xi_{I} \lambda_{IK}~ e_{K} e_{I}$, where $\lambda_{IK} = 0$ if $[e_{K} , e_{I}] = 0$ and $\lambda_{IK} = 2$ otherwise.  Multiplying on the left by $e_{K}$ yields $0 = \sum_{I \in \fI_{\fK}} \eta_{IK} e_{I}$, where $\eta_{IK} = \xi_{I} \lambda_{IK} e_{K}^{2}$.   By the linear independence of $\{e_{I} : I \in \fI \}$ we conclude that $\eta_{IK} = 0$ for all $I \in \fI_{\fK}$.  Suppose now that $\xi_{I} \neq 0$ for some $I \in \fI_{\fK}$.  Then $\lambda_{IK} = 0$ since $e_{K}^{2} \neq 0$ by Proposition 1.3.  This means that $[e_{K},e_{I}] = 0$ and since $e_{K} \in \fK$ was arbitrary we conclude  that $e_{I} \in \fZ(\fK_{r,s})$.
\end{proof}

\noindent  $\mathit{Proof~of~ Proposition~4.2}$  We need some preliminary results

\begin{lemma}  Let $e_{I} \in \fZ(\fK_{r,s})$ for some $I \in \fI$.  Let $I \cap \{1,2,...,r \}$ be nonempty.  Then $ \{1,2,...,r \} \subset I$.
\end{lemma}

\begin{proof}  Let $i \in I \cap \{1,2,...,r \}$.  If $J = \{e_{i} \}$, then $|J| \equiv 1~(mod~4)$ and $J^{+}$ is empty.  Hence $e_{i} \in \fG_{r,s}$ by Proposition 1.5, and  $e_{i} \in \fK_{r,s}$ by Lemma 2.7 or 2.8 and the definition of $\fK_{r,s}$.  It follows that $e_{i} e_{I} = e_{I} e_{i}$ since $e_{I} \in \fZ(\fK_{r,s})$.  By part 2) of Proposition 3.2 we conclude that $|I|$ is odd since $|I \cap \{i \}| = |\{ i\}| = 1$.  Next suppose that $\{1,2,...,r \} - I$ is nonempty and let $k \in \{1,2,...,r \} - I$.  If $K = \{k \}$, then $I \cap K$ is empty.  As above, $e_{k} \in \fK$ and hence $e_{k} e_{I} = e_{I} e_{k}$.  Since $|I|$ and $|K|$ are odd and $|I \cap K|$ is even we obtain a contradiction to part 1) of Proposition 3.2.  Hence  $ \{1,2,...,r \} \subset I$
\end{proof}

\begin{lemma}  Let $e_{I} \in \fZ(\fK_{r,s})$ for some $I \in \fI$.  Let $I \cap \{r+1,r+ 2,...,r +s\}$ be nonempty.  Then $ \{r+1,r+2,...,r+s \} \subset I$.  
\end{lemma}

\begin{proof} By hypothesis  there exists $\alpha \in I^{+} = I \cap \{r+1,r+ 2,...,r +s\}$.  Suppose there exists $\beta \in \{r+1,r+ 2,...,r +s\} - I$ and let $J = \{\alpha, \beta \}$.  By Proposition 1.5 we see that $e_{J} \in \fG$ since $|J| \equiv 2~(mod ~4)$, and moreover $e_{J} \in \fK$ by Lemma 2.8 since $|J| = |J^{+}| = 2$.  Hence $e_{I} e_{J} = e_{J} e_{I}$.  On the other hand $|I \cap J| = |\{ \alpha\}| = 1$ and $|J|$ is even, which contradicts part 2) of Proposition 3.2.  We conclude that $ \{r+1,r+2,...,r+s \} \subset I$.  
\end{proof}

\begin{lemma}  Let $e_{I} \in \fZ(\fK_{r,s})$ for some $I \in \fI$.  Then one of the following holds :

	1) $I = \{1,2,..., r \}$ 
	
	2)  $I = \{r+1, r+2, ... , r+s \}$ 
	
	3)  $I = \{1,2,..., r, r+1, r+2, ... , r+s \}$.
\end{lemma}

\begin{proof}  We consider the various possibilities for $I^{+} = I \cap \{r+1,r+2, ... , r+s \}$ and $I^{-} = I \cap \{1,2,..., r \}$.  If $I^{+}$ and $I^{-}$ are both nonempty, then  $I = \{1,2,..., r, r+1, r+2, ... , r+s \}$ by Lemmas 4.4 and 4.5.  If $I^{+}$ is nonempty, and $I^{-}$ is empty, then $I = \{r+1, r+2, ... , r+s \}$ by Lemma 4.5.  If $I^{-}$ is nonempty, and $I^{+}$ is empty, then $I = \{1,2,..., r \}$ by Lemma 4.4.  
\end{proof}

\noindent We have shown that possibilities 1), 2) and 3) of Lemma 4.6 are the only ones that could arise if $e_{I}$ is to lie in the center of $\fK_{r,s}$.  We now complete the proof of Proposition 4.2 by showing under what conditions these three possibilities do arise. 
\newline 

\begin{lemma}  Let $I = \{1,2,..., r \}$.  Then $e_{I} \in \fZ(\fK_{r,s}) \Leftrightarrow r \equiv 1~(mod~4)$
\end{lemma} 

\begin{proof}  Suppose first that $e_{I} \in \fZ(\fK_{r,s})$.  Then $r = |I| \equiv 1$ or $2~(mod~4)$ by 2) of Proposition 1.5  since $e_{I} \in \fG_{r,s}$.  We suppose that $r \equiv 2~(mod~4)$ and obtain a contradiction.  In particular $|I| = r$ is even.  Let J be any subset of $\{1,2,...,r \}$ with $|J| \equiv 1~(mod~4)$.  Then $e_{J} \in \fG_{r,s}$ by Proposition 1.5, and furthermore $e_{J} \in \fK_{r,s}$ by  Lemma 2.8 since $|J^{+}| = 0$.  Hence $e_{I} e_{J} = e_{J} e_{I}$.  On the other hand $|I|$ is even and $|I \cap J| = |J|$ is odd, so $e_{I}$ anti-commutes with $e_{J}$ by part 2) of Proposition 3.2.  We conclude that $r \equiv 1~(mod~4)$.
\newline

\noindent  Next we show that $e_{I} \in \fZ(\fK_{r,s})$ if $r \equiv 1~(mod~4)$.  In this case $r = |I|$ is odd.  Let $J \in \fI$ with $e_{J} \in \fK_{r,s}$ be given.  It suffices to show that $e_{I}$ commutes with both $e_{J^{+}}$ and $e_{J^{-}}$ since $e_{J} = \pm e_{J^{+}} e_{J^{-}}$.  The fact that $e_{J} \in \fK_{r,s}$ implies that $|J^{+}|$ is even by Lemma 2.8.  Since $I$ and $J^{+}$ are disjoint it follows from Proposition 3.1 that $e_{I} e_{J^{+}} = (-1)^{|I| |J^{+}|} e_{J^{+}} e_{I} = e_{J^{+}} e_{I}$ since $|J^{+}|$ is even.
\newline

\noindent  Next we show that $e_{I}$ commutes with $e_{J^{-}}$.  Note that $I \cap J^{-} = J^{-}$.  If $|J^{-}|$ is even, then since $|I|$ is odd it follows from part 1) of Proposition 3.2 that $e_{I} e_{J^{-}} = e_{J^{-}} e_{I}$.  If $|J^{-}|$ is odd, then $e_{I} e_{J^{-}} = e_{J^{-}} e_{I}$ by part 2) of Proposition 3.2.  This completes the proof of the lemma.
\end{proof}

\begin{lemma}  Let $I = \{r+1,r+2,..., r+s \}$.  Then $e_{I} \in \fZ(\fK_{r,s}) \Leftrightarrow s \equiv 2~(mod~4)$
\end{lemma}

\begin{proof}  Suppose first that $e_{I} \in \fZ(\fK_{r,s}) \subset \fK_{r,s}$.  Then $|I| = |I^{+}|$ is even by Lemma 2.8.  Since $e_{I} \in \fG_{r,s}$ it follows from 2) of Proposition 1.5 that $|I| \equiv 1$ or $2~(mod~4)$.  We conclude that $s = |I| \equiv 2~(mod~4)$. 
\newline

\noindent  Conversely, suppose that $s \equiv 2~(mod~4)$.  Then $e_{I} \in \fG_{r,s}$ by 2) of Proposition 1.5, and moreover $e_{I} \in \fK_{r,s}$ by Lemma 2.8 since $s = |I| = |I^{+}|$ is even.  Now let $e_{J} \in \fK_{r,s}$.  Then $|I \cap J^{+}| = |J^{+}|$ is even by Lemma 2.8.  It follows that $e_{I}$ commutes with $e_{J^{+}}$ by part 1) of Proposition 3.2.  Since $I \cap \{1,2,..., r \}$ is empty it follows from Proposition 3.1 that $e_{I} e_{J^{-}} = (-1)^{|I| |J^{-}|} e_{J^{-}} e_{I} = e_{J^{-}} e_{I}$ since $|I|$ is even.  Hence $e_{I}$ commutes with $e_{J} = \pm e_{J^{+}} e_{J^{-}}$.  We conclude that $e_{I} \in \fZ(\fK_{r,s})$ since $e_{J} \in \fK_{r,s}$ was arbitrary.
\end{proof}

\begin{lemma}  Let $I = \{1,2,...,r, r+1, r+2, ..., r+s \}$.  Then $e_{I} \in \fZ(\fK_{r,s}) \Leftrightarrow a)~ r+s \equiv 1~(mod~4)$ and b)~ s is even.
\end{lemma}

\begin{proof}  Suppose first that $e_{I} = \omega \in \fZ(\fK_{r,s})$.  Then $\omega \in \fG_{r,s}$ and it follows from Proposition 1.5 that $r+s \equiv 1$ or $2~(mod~4)$.  Recall that $e_{1} \in \fK_{r,s}$ by Lemma 2.7 since $e_{1}^{2} = -1$.  If $r+s$ is even, then $e_{1} \omega = \omega \alpha(e_{1}) = - \omega e_{1}$ by Lemma 5.1 of [E], a contradiction.  Hence $r+s \equiv 1~(mod~4)$.  Moreover, since $\omega \in \fK_{r,s}$ it follows from Lemma 2.8 that $s = |\omega^{+}|$ is even.
\newline

\noindent  Conversely, suppose that $r+s \equiv 1~(mod~4)$ and s is even.  By Proposition 1.6 it follows that $\omega$ lies in the center of $C\ell(r,s)$.  Furthermore $\omega \in \fG_{r,s}$ by Proposition 1.5 since $r+s \equiv 1~(mod~4)$.  Finally, $\omega \in \fK_{r,s}$ by Lemma 2.8 since $|\omega^{+}| = s$ is even.  We conclude that $\omega \in \fZ(\fK_{r,s})$.
\end{proof}

\section{Equivalence of bilinear forms and field extensions}

\noindent Let V be a finite dimensional vector space over a field F of characteristic $\neq 2$, and let $Q_{i} : V \times V \rightarrow F$ be symmetric bilinear forms for $i = 1,2$.  The bilinear forms $Q_{1}, Q_{2}$ are said to be $\mathit{equivalent}$ if there exists a nonsingular linear transformation $T : V \rightarrow V$ such that $Q_{2}(v,w) = Q_{1}(T(v), T(w))$ for all $v,w \in V$.  It is easy to see that being equivalent is an equivalence relation on the vector space of bilinear forms on V.

\begin{proposition}  Let V be a finite dimensional vector space over F, and let $Q_{1}, Q_{2}$ be equivalent, nondegenerate, symmetric, bilinear forms on V.  Then $C\ell(V,Q_{1})$ is algebra isomorphic to $C\ell(V,Q_{2})$.
\end{proposition}

\begin{proof}  By the equivalence of $Q_{1}$ and $Q_{2}$ there exists  a nonsingular linear transformation $T : V \rightarrow V$ such that $Q_{2}(v,w) = Q_{1}(T(v), T(w))$ for all $v,w \in V$.  Define $i : V \rightarrow C\ell(V, Q_{1})$ by $i(v) = T(v)$.  Clearly i is linear and injective.  Moreover, if $v \in V$, then $- Q_{2}(v,v) = - Q_{1}(T(v),T(v))= T(v)^{2} = i(v)^{2}$.  Hence by the universal mapping definition of $C\ell(V,Q_{1})$ (Proposition 1.1) the injective linear map i extends to an algebra homomorphism $\tilde{i} : C\ell(V, Q_{2}) \rightarrow C\ell(V, Q_{1})$.  The map $\tilde{i}$ is surjective since $\tilde{i}(V) = i(V) = V$ generates $C\ell(V,Q_{1})$ as an algebra (see for example Corollary 1.2 of [E] for a proof).  We conclude that the map $\tilde{i}$ is an isomorphism since the dimensions of $C\ell(V, Q_{1})$ and $C\ell(V, Q_{2})$ are the same.
\end{proof}

\noindent $\mathbf{Field~extensions}$

\noindent  Let V be a finite dimensional vector space over a field F of characteristic $\neq 2$, and let $\overline{F}$ be a field that is a finite extension of F.  Let Q be a nondegenerate, symmetric, bilinear form on V.  If $\overline{V} = V \otimes_{F} \overline{F}$, then we shall extend Q to a nondegenerate, symmetric, bilinear form $\overline{Q}$  on $\overline{V}$.
\newline

\noindent  Let $\fB = \{v_{1}, ... , v_{n} \}$ be a Q-orthogonal basis of V.  Let $\overline{Q} : \overline{V} \times \overline{V} \rightarrow \overline{F}$ be the unique symmetric, bilinear form such that $\overline{Q}(v_{i} \otimes 1, v_{j} \otimes 1) = Q(v_{i}, v_{j})$ for $1 \leq i,j \leq n$.  The bilinearity of Q and $\overline{Q}$ implies that $\overline{Q}(v \otimes 1, w \otimes 1) = Q(v,w)$ for all $v,w \in V$.  In particular,  $\overline{Q}$ does not depend on the basis $\fB$ of V.  Hence $\{v_{1} \otimes 1, ... , v_{n} \otimes 1 \}$ is a $\overline{Q}$-orthogonal  basis of $\overline{V}$.  Moreover, since Q is nondegenerate, we have  $\overline{Q}(v_{i} \otimes 1, v_{i} \otimes 1) = Q(v_{i}, v_{i}) \neq 0$ for $1 \leq i \leq n$.  It follows that $\overline{Q}$ is nondegenerate on $\overline{V}$.
\newline

\noindent  Next, we show that $C\ell(V,Q) \otimes \overline{F}$ is an algebra over $\overline{F}$.  Let $\fB = \{v_{1}, ... , v_{n} \}$ be a basis of V.  Then $\overline{\fB} = \{v_{1} \otimes 1, ... , v_{n} \otimes 1 \}$ is a basis for $\overline{V} = V \otimes \overline{F}$.  We define $(v_{i} \otimes 1) \cdot (v_{j} \otimes 1) = (v_{i} \cdot v_{j}) \otimes 1$ for $1 \leq i,j \leq n$ and extend this to a multiplication on $C\ell(V,Q) \otimes \overline{F}$ by bilinearity.  In particular if $v,w \in V$, then $(v \otimes 1) \cdot (w \otimes 1) = (v \cdot w) \otimes 1$, so the multiplication on $C\ell(V,Q) \otimes \overline{F}$ does not depend on the choice of basis $\fB$ of V.
 
\begin{proposition}  The algebra  $C\ell(V,Q) \otimes \overline{F}$ is algebra isomorphic to $C\ell(\overline{V}, \overline{Q})$.
\end{proposition}

\begin{proof} Let $i : \overline{V} \rightarrow (C\ell(V,Q) \otimes \overline{F} , \overline{Q})$ be the inclusion map.  We use Proposition 1.1 to show that i extends to an algebra homomorphism $\overline{i} : C\ell(\overline{V}, \overline{Q}) \rightarrow (C\ell(V,Q) \otimes \overline{F}, \overline{Q})$.  We then show that $\overline{i}$ is surjective and use a dimension argument to conclude that $\overline{i}$ is an algebra isomorphism.
\newline

\noindent Let $\{v_{1}, ... , v_{n} \}$ be a Q-orthogonal basis of V, and let $\{v_{1} \otimes 1, ... , v_{n} \otimes 1 \}$ be the corresponding $\overline{Q}$-orthogonal  basis of $\overline{V}$.  Let $\xi \in \overline{V}$ be given and write $\xi = \sum_{i=1}^{n} \alpha_{i} (v_{i} \otimes 1) = \sum_{i=1}^{n} (v_{i} \otimes \alpha_{i})$, where $\alpha_{i} \in \overline{F}$ for all i.  Then $\xi^{2} = \sum_{i,j = 1}^{n} (v_{i} \otimes \alpha_{j}) \cdot (v_{j} \otimes \alpha_{j}) = \newline \sum_{i < j} (v_{i} \cdot v_{j} + v_{j} \cdot v_{i}) \otimes \alpha_{i} \alpha_{j} + \sum_{i=1}^{n} v_{i}^{2} \otimes \alpha_{i}^{2} = \sum_{i=1}^{n} v_{i}^{2} \otimes \alpha_{i}^{2}$ by the Q-orthogonality of $\{v_{1}, ... , v_{n} \}$.  Similarly, $ - \overline{Q} (\xi , \xi) = - \sum_{i,j = 1}^{n} \alpha_{i} \alpha_{j} ~ \overline{Q}(v_{i} \otimes 1, v_{j} \otimes 1) = - \sum_{i=1}^{n} \alpha_{i}^{2}~\overline{Q}(v_{i} \otimes 1, v_{i} \otimes 1) = \sum_{i=1}^{n} \alpha_{i}^{2} (v_{i} \otimes 1) \cdot (v_{i} \otimes 1) = \sum_{i=1}^{n} (v_{i} \otimes \alpha_{i}) \cdot (v_{i} \otimes \alpha_{i}) = \sum_{i=1}^{n} v_{i}^{2} \otimes \alpha_{i}^{2} = \xi^{2} = i(\xi) \cdot i(\xi)$.   It now follows from Proposition 1.1 that $i : \overline{V} \rightarrow C\ell(V,Q) \otimes \overline{F} , \overline{Q})$ extends to an algebra homomorphism $\overline{i} : C\ell(\overline{V}, \overline{Q}) \rightarrow (C\ell(V,Q) \otimes \overline{F}, \overline{Q})$.
\newline

\noindent The homomorphism $\overline{i}$ is surjective since $\overline{i}(\overline{V}) = i(\overline{V})$ generates $C\ell(V,Q) \otimes \overline{F}$ as an algebra.  We conclude that $\overline{i}$ is an isomorphism since $C\ell(V,Q) \otimes \overline{F}$ and $C\ell(\overline{V}, \overline{Q})$ both have dimension $2^{n}$ over $\overline{F}$.  
\end{proof}

\begin{corollary}  Let V be a finite dimensional vector space over $\br$.  Let $Q_{1}$ and $Q_{2}$ be nondegenerate, symmetric bilinear forms on V.  Let $\overline{Q_{1}}$ and $\overline{Q_{2}}$ be their extensions to nondegenerate, symmetric, bilinear forms on $\overline{V} = V \otimes \bc$.  Then $\overline{Q_{1}}$ and $\overline{Q_{2}}$ are equivalent.
\end{corollary}

\begin{proof}  It suffices to show that any nondegenerate, symmetric, bilinear form $\overline{Q}$ on $\overline{V}$ has a basis $\fB = \{ \overline{v_{1}}, ... , \overline{v_{n}} \}$ such that $\overline{Q}(\overline{v_{i}}, \overline{v_{j}}) = \delta_{ij}$.  Let $\{w_{1}, ... , w_{n} \}$ be a $\overline{Q}$-orthogonal basis of $\overline{V}$.  Note that $\overline{Q}(w_{i}, w_{i}) \neq 0$ for all i by the nondegeneracy of $\overline{Q}$.  Since $\bc$ is algebraically closed we may write $x^{2} - \frac{1}{\overline{Q}(w_{i}, w_{i})} = (x - \alpha_{i})(x - \beta_{i})$ for suitable $\alpha_{i}, \beta_{i} \in \bc$ and $ 1 \leq i \leq n$.  If $\overline{v_{i}} = \alpha_{i} w_{i}$, then $\overline{Q}(\overline{v_{i}}, \overline{v_{i}}) = 1$ for all i, and we conclude that $\overline{Q}(\overline{v_{i}}, \overline{v_{j}}) = \delta_{ij}$ for all i,j.
\end{proof}

\begin{corollary} Let r,s be nonnegative integers with $r+s > 0$.  Then

	1)  $C\ell(r,s) \otimes \bc \approx C\ell(r+s,0) \otimes \bc$.

	2)  $\fG_{r,s} \otimes \bc \approx \fG_{r+s,0} \otimes \bc$.
\end{corollary}

\begin{proof}  Assertion 1) follows immediately from Proposition 5.2 and Corollary 5.3.  We prove 2).  If $c_{r,s}$ denotes the canonical anti-automorphism of $C\ell(r,s)$, then $c_{r,s} \otimes \bc$ is the canonical anti-automorphism of $C\ell(r,s) \otimes \bc$ since $(c_{r,s} \otimes \bc)(e_{i} \otimes 1) = - e_{i} \otimes 1$ for $1 \leq i \leq r+s$.  Similarly, if $c_{r+s,0}$ is the canonical anti-automorphism of $C\ell(r+s,0)$, then $c_{r+s,0} \otimes \bc$ is the canonical anti-automorphism of $C\ell(r+s,0) \otimes \bc$.  Identifying $C\ell(r,s) \otimes \bc$ with $C\ell(r+s,0) \otimes \bc$ yields

	$(^{*}) c_{r,s} \otimes \bc = c_{r+s,0} \otimes \bc$.
	
\noindent  Now observe that $\fG_{r,s}$ is the $-1$-eigenspace of $c_{r,s}$ in $C\ell(r,s)$.  Hence $\fG_{r,s} \otimes \bc$ is the $-1$-eigenspace of $c_{r,s} \otimes \bc$ in $C\ell(r,s) \otimes \bc = C\ell(r+s,0) \otimes \bc$.  Similarly, $\fG_{r+s,0} \otimes \bc$ is the $-1$-eigenspace of $c_{r+s,0} \otimes \bc$ in $C\ell(r+s,0) \otimes \bc$.  Assertion 2) of the Proposition now follows from $(^{*})$.
\end{proof}

\begin{corollary}  Let r,s be nonnegative integers with $r+s > 0$.  Let $\fH_{r,s}$ denote the semisimple ideal of $\fG_{r,s}$ as defined in Proposition 1.7 and Proposition 6.12 of [E].  Then $\fH_{r,s} \otimes \bc$ is given by the following table

	$\mathbf{r+s \hspace{2.2in} \fH_{r,s} \otimes \bc}$
\newline

	8k \hspace{2.2in} $\fs \fo(2^{4k}, \br) \otimes \bc$
	
	8k+1 \hspace{2in} $\fs \fu(2^{4k}) \otimes \bc$
	
	8k+2 \hspace{2in} $\fs \fp(2^{4k}) \otimes \bc$
	
	8k+3 \hspace{2in} $(\fs \fp(2^{4k}) \times \fs \fp(2^{4k})) \otimes \bc$
	
	8k+4 \hspace{2in} $\fs \fp(2^{4k+1}) \otimes \bc$
	
	8k+5 \hspace{2in} $\fs \fu(2^{4k+2}) \otimes \bc$
	
	8k+6 \hspace{2in} $\fs \fo(2^{4k+3}, \br) \otimes \bc$
	
	8k+7 \hspace{2in} $(\fs \fo(2^{4k+3}, \br) \times \fs \fo(2^{4k+3}, \br)) \otimes \bc$
\end{corollary}

\begin{proof}  This follows immediately from Theorem 9.5 of [E] and 2) of Corollary 5.4.  We also note that a) $\fu(n) \approx \fs \fu(n) \times \br$ in the cases where $r+s \equiv 1~(mod~4)$ and  b) $\fH_{r,s}$ has codimension 1 in $\fG_{r,s}$ if $r+s \equiv 1~(mod~4)$ and $\fH_{r,s} = \fG_{r,s}$ otherwise (Proposition 6.12) of [E]).
\end{proof}

\section{A method for computing $\fH_{r,s}$}

\noindent  In what follows we assume that $F = \br$. By Proposition 1.7 we may write  $\fG_{r,s} = \fZ(\fG) \oplus \fH_{r,s}$, where $\fH_{r,s}$ is a semisimple ideal of $\fG_{r,s}$.  If $r+s \equiv 1~(mod~4)$, then $\fZ(\fG) = \br~\omega$ and $\fH_{r,s}$ is a codimension 1 ideal of $\fG_{r,s}$.  If $r+s \neq 1~(mod~4)$, then $\fG_{r,s} = \fH_{r,s}$. 
\newline

\noindent  In this section we develop a method for computing $\fH_{r,s}$, and we give a brief outline here.

	 1) We define $\fK^{\prime}_{r,s} = \fK_{r,s} \cap \fH_{r,s}$.  We show that $\fH_{r,s} = \fK_{r,s}^{\prime} \oplus \fP_{r,s}$.
	 	 
	 2) We consider each of the  cases $r+s = 8k + \alpha, 0 \leq \alpha \leq 7$ and $ s \equiv \beta~(mod~4), 0 \leq \beta \leq 3$.  Then we proceed as follows :
	 
\hspace{.2in}	a)  We compute the dimension of $\fK^{\prime}_{r,s}$, using Propositions 2.13 and 6.1. 
		
\hspace{.2in}	b)  We compute the dimension of the center of $\fK^{\prime}_{r,s}$ using Propositions 4.1 and 4.2 and Corollary 6.3 when $r+s \equiv 1~(mod~4)$.  
		
\hspace{.2in}	c)  If $r+s$ is not congruent to 3 (mod 4), then $\fH_{r,s}$ is a real form of $\fU^{\bc}$, where $\fU$ is a compact simple Lie algebra , and $\fK^{\prime}_{r,s}$  is Lie algebra isomorphic to the fixed point set of some involutive automorphism $\tau$ of $\fU$.
\newline 

\noindent The number of conjugacy  classes in $Aut(\fU)$ of involutive automorphisms  $\tau$ of $\fU$ is either two or three, depending on $\fU$, and a complete list of these conjugacy classes is known.   See for example [He, pp. 451-455] for a description.
			 
\hspace{.2in}	d)  If $r+1 \equiv 3,~mod~4$, then $\fH_{r,s}$ is a real form of  $\fU^{\bc} \oplus \fU^{\bc}$, where $\fU$ is a compact simple Lie algebra.  The  method needed here is a more complicated version of  c).
\newline 

\noindent For each case $r+s = 8k + \alpha, 0 \leq \alpha \leq 7$ and $s \equiv \beta~(mod~4), 0 \leq \beta \leq 3$ the Lie algebra $\fH_{r,s}$ can now be determined using the steps above.  In some cases step 2b) may be omitted.
\newline

\noindent $\mathbf{Definition}$  Let $\fK^{\prime}_{r,s} = \fK_{r,s} \cap \fH_{r,s}$. If $r+s \neq 1~(mod~4)$, then clearly $\fK^{\prime}_{r,s} = \fK_{r,s}$ since $\fH_{r,s} = \fG_{r,s}$.
\begin{proposition}  Let $r+s \equiv 1~(mod~4)$.  Then

	1)  If $\omega \in \fK_{r,s}$, then $\fK_{r,s} = \br \omega \oplus \fK^{\prime}_{r,s}$.
	
	2)  If $\omega \notin \fK_{r,s}$, then $\fK_{r,s} = \fK^{\prime}_{r,s}$.
\end{proposition} 

\noindent $\mathbf{Remark}$ Since $s = |\omega^{+}|$ it follows from Lemma 2.8  that $\omega \in \fK_{r,s} \Leftrightarrow $ s is even.
\newline

\noindent We need a preliminary result.

\begin{lemma}  The automorphism $\beta$ of $\fG_{r,s}$ leaves invariant $\fH_{r,s}$.  
\end{lemma}

\begin{proof}  By Propositions 1.5, 1.6 and 1.7  $\fH_{r,s} = $span-$\{e_{I} : I \neq \{1,2,...,r+s \}~ \rm{and}~e_{I} \in \fG_{r,s} \}$.  The assertion now follows immediately since $\beta(e_{I}) = \pm e_{I}$ for all $I \in \fI$.
\end{proof}

\noindent  We now prove 1).  Let $\omega \in \fK_{r,s}$.  Then $s = |\omega^{+}|$ must be even by Lemma 2.8.  It follows that $\beta(\omega) = \omega$ since $\beta(\omega) = (-1)^{s} \omega$ by the definitions of $\beta$ and $\omega$.  Clearly we have 

	$(^{*})$ \hspace{.2in} $\br \omega \oplus \fK^{\prime}_{r,s} \subseteq \fK_{r,s}$.
	
\noindent  It remains to prove that the inclusion in $(^{*})$ is an equality.  Let $\xi \in \fK_{r,s}$ be given.  We may choose elements $\alpha \in \br$ and $\xi^{\prime} \in \fH_{r,s}$ such that $\xi = \alpha \omega + \xi^{\prime}$ since $\fG_{r,s} = \br \omega \oplus \fH_{r,s}$ when $r + s \equiv 1~(mod~4)$ by Propositions 1.6 and 1.7.  Recall that $\fK_{r,s}$ is the +1 eigenspace of $\beta|_{\fG_{r,s}}$.  Hence $\alpha \omega + \xi^{\prime} = \xi = \beta(\xi) = \alpha \beta(\omega) + \beta(\xi^{\prime}) = \alpha \omega + \beta(\xi^{\prime})$.  We conclude that $\beta(\xi^{\prime}) = \xi^{\prime}$ and hence $\xi^{\prime} \in \fK_{r,s} \cap \fH_{r,s} = \fK^{\prime}_{r,s}$.  This shows that equality holds above in $(^{*})$.
\newline

\noindent  We now prove 2).  Note that $\omega \in \fG_{r,s}$ by 2) of Proposition 1.5  since $r+s \equiv 1~(mod~4)$.  Suppose now that $\omega \notin \fK_{r,s}$.  Then $s = |\omega^{+}|$ is odd by Lemma 2.8, and it follows that $- \omega = (-1)^{s} (\omega) = \beta(\omega)$.  Clearly it suffices to show that $\fK_{r,s} \subseteq \fK^{\prime}_{r,s}$.  Let $\xi \in \fK_{r,s}$ be given and write $\xi = \alpha \omega + \xi^{\prime}$ for suitable elements $\alpha \in \br$ and $\xi^{\prime} \in \fH_{r,s}$.  Then $\alpha \omega + \xi^{\prime} = \xi = \beta(\xi) = \alpha \beta(\omega) + \beta(\xi^{\prime}) = - \alpha \omega + \beta(\xi^{\prime})$.  Hence $2 \alpha \omega = \beta(\xi^{\prime}) - \xi^{\prime}$, which lies in $\br \omega \cap \fH_{r,s} = \{0 \}$.  It follows that $\alpha = 0$ and $\xi = \xi^{\prime} \in \fH_{r,s}$.  We conclude that $\xi \in \fK_{r,s} \cap \fH_{r,s} = \fK^{\prime}_{r,s}$.  The proof of 2) is complete.

\begin{corollary}  Let $r+s \equiv 1~(mod~4)$.  Then

	1)  If $\omega \in \fK_{r,s}$, then $dim ~\fZ(\fK_{r,s}) = 1 + dim~ \fZ(\fK_{r,s}^{\prime})$.
	
	2)  If $\omega \notin \fK_{r,s}$, then $\fZ(\fK_{r,s}) = \fZ(\fK_{r,s}^{\prime})$.
\end{corollary}

\begin{proof}  This follows immediately from Propositions 1.6 and 6.1.
\end{proof}

\begin{proposition} $\fH_{r,s} = \fK^{\prime}_{r,s} \oplus \fP_{r,s}$.
\end{proposition}

\begin{proof}  This is an immediate consequence of the facts that $\fG_{r,s} = \fK_{r,s} \oplus \fP_{r,s}$ and $\fK^{\prime}_{r,s} = \fK_{r,s} \cap \fH_{r,s}$.
\end{proof}

\section {Involutive automorphisms and noncompact real forms}

\noindent A real Lie algebra $\fU$ is said to be $\mathit{compact}$ if the Killing form is negative definite on $\fU$.   This implies that every Lie group U with Lie algebra $\fU$ must be compact (cf. Proposition 6.6, Corollary 6.7 and Theorem 6.9 in chapter II of [H]).   Let $\fU$ be a compact,  real Lie algebra, and let $\tau$ be a nontrivial automorphsim of $\fU$.  Let $\fK_{0}$ and $\fP^{*}$ be the +1 and -1 eigenspaces of $\tau$ respectively.   Define a real Lie algebra $\fG_{0} = \fK_{0} \oplus \fP_{0} \subset \fU^{\bc}$, where $\fP_{0} = i \fP^{*}$.  Since $\fU = \fK_{0} \oplus \fP^{*}$ is compact the Killing form of $\fU$ is negative definite .  It follows that the Killing form of $\fG_{0}$ is negative definite on $\fK_{0}$ and positive definite on $\fP_{0}$.  By inspection $\fG_{0}^{\bc} = \fU^{\bc}$ and hence $\fG_{0}$ is a noncompact real form of $\fU^{\bc}$. 
\newline

\noindent The work of E. Cartan shows that if $\fU$ is a compact, $\mathit{simple}$ real Lie algebra, then every noncompact real form $\fG_{0}$ of $\fU^{\bc}$ arises from a suitable involutive automorphism $\tau$ of $\fU$ (cf. [He, pp.451-455])  For convenience later we say that $\fG_{0} = \fK_{0} \oplus \fP_{0}$ is $\mathit{induced}$ by $\tau$.
\newline

\noindent $\mathbf{Remark}$  Let $\tau$ be an involutive automorphism of a compact Lie algebra $\fU$, and let $\fG$ be the Lie algebra induced by $\tau$.  Let $\varphi$ be an automorphism of $\fU$, and let $\tau^{\prime} = \varphi \circ \tau \circ \varphi^{-1}$.  Then $\tau^{\prime}$ is also an involutive automorphism of $\fU$ that induces $\fG$.  This is a consequence of the fact that $\varphi$ maps the $+1$ (respectively $-1$)  eigenspace of $\tau$ onto the  $+1$ (respectively $-1$)  eigenspace of $\tau^{\prime}$.

\begin{proposition}  Let $\fG$ be a real semisimple Lie algebra.  Let $\tau_{1}, \tau_{2}$ be involutive automorphisms of compact Lie algebras $\fU_{1}, \fU_{2}$ that both induce $\fG$.  Then

	1)  $\fU_{1}$ is Lie algebra isomorphic to $\fU_{2}$
	
	2)  $\fK_{1} = Fix(\tau_{1})$ is Lie algebra isomorphic to $\fK_{2} = Fix(\tau_{2})$.
\end{proposition} 

\begin{proof}  Observe that $\fU_{1}^{\bc} = \fG^{\bc} = \fU_{2}^{\bc}$.  Assertion 1) now follows from the uniqueness of compact real forms ( see for example Corollary 7.3, chapter III of [He]).  To prove 2) we may assume that $\fU_{1} = \fU_{2} = \fU$.  Let $\fK_{j}, \fP_{j}^{*}$ denote the $+1, -1$-eigenspaces of $\tau_{j}$ for $j = 1,2$.  Then $\fG = \fK_{j} \oplus \fP_{j} \subset \fU^{\bc}$ for $j = 1,2$, where $\fP_{j} = i~ \fP_{j}^{*}$.  The Killing form of $\fU$ is negative definite, and hence the Killing form of $\fG$ is positive definite on $\fP_{j}$ and negative definite on $\fK_{j}$ for $i = 1,2$.  The subspace $\fK_{j}$ is a subalgebra of $\fG$ since it is the fixed point set of the automorphism $\tau_{j}$ for $j = 1,2$.  Hence $\fG = \fK_{j} \oplus \fP_{j}$ is a $\mathit{Cartan~decomposition}$ of $\fG$ for $j = 1,2$ (cf. Proposition 7.4, chapter III of [He]).  Assertion 2) now follows from Theorem 7.2, chapter III of [He]. 
\end{proof}

\begin{proposition}  There exists a compact Lie algebra $\fU$ and an involutive automorphism $\tau$ of $\fU$ such that 

	1)  $\tau$ induces $\fH_{r,s}$.
	
	2)  $Fix(\tau) \approx \fK^{\prime}_{r,s}$
\end{proposition}

\begin{proof}  Let $\beta : C\ell(r,s) \rightarrow C\ell(r,s)$ be the automorphism constructed in Proposition 2.1 that leaves invariant $\fG_{r,s}$.  If $I = (i_{1}, ... , i_{k})$ is an arbitrary multi-index, then by inspection $\beta(e_{I}) = \pm e_{I}$.  Hence $\beta$ leaves invariant $\fH_{r,s}$ by the construction of $\fH_{r,s}$ in Proposition 6.12 of [E].  Recall from Proposition 6.3 that  $\fH_{r,s} = \fK^{\prime}_{r,s} \oplus \fP_{r,s}$, where $\fK^{\prime}_{r,s}, \fP_{r,s}$ are the $+1, -1$-eigenspaces of $\beta$ in $\fH_{r,s}$.  Let $\fU = \fK^{\prime}_{r,s} \oplus i~\fP_{r,s} \subset \fG^{\bc}$.  By Proposition 2.5 above and Corollary 6.10 of [E]  the Killing form of $\fH_{r,s}$ is positive definite on $\fP_{r,s}$ and negative definite on $\fK^{\prime}_{r,s}$, and it follows that the Killing form of $\fU$ is negative definite.  The fact that $\beta$ is an involutive automorphism of $\fH_{r,s}$ implies that $ [\fK^{\prime}_{r,s}, \fK^{\prime}_{r,s}] \subset \fK^{\prime}_{r,s}, [\fP_{r,s}, \fP_{r,s}] \subset \fK^{\prime}_{r,s}$ and $[\fK^{\prime}_{r,s}, \fP_{r,s}] \subset \fP_{r,s}$.  Similar bracket relations hold between $\fK^{\prime}_{r,s}$ and $i~\fP_{r,s}$, and it follows there exists an involutive automorphism $\tau$ of $\fU$ whose $+1,-1$-eigenspaces are $\fK^{\prime}_{r,s}, i~\fP_{r,s}$.  By inspection assertion 1)  follows immediately.  Assertion 2) follows since $Fix(\tau) \approx Fix(\beta) = \fK^{\prime}_{r,s}$.
\end{proof}

\section{The case that $r+s \equiv~3~(mod~4)$}

\noindent  For notational convenience we adopt the notation $\fG^{\bc}$ to denote $\fG \otimes \bc$.  Note that $\fG_{r,s} = \fH_{r,s}$ when $r+s \equiv~3~(mod~4)$ by Propositions 1.6 and 1.7.

\begin{proposition}  Let $r+s \equiv 3~mod~4$.  Then

	1)  If s is even, then there exist simple ideals $\fG_{1}, \fG_{2}$ of $\fG_{r,s}$ and a compact, simple Lie algebra $\fU$ such that $\fG_{r,s} = \fG_{1} \oplus \fG_{2}$ and $\fG_{1}^{\bc} = \fG_{2}^{\bc} = \fU^{\bc}$ is a complex simple Lie algebra.
	
	2)  If s is odd, then $\fG_{r,s}$ is simple and is the realification of a complex simple Lie algebra.
\end{proposition}

\noindent $\mathbf{Remarks}$

	1)   Note that $\omega$ lies in the center of $C\ell(r,s)$ by Proposition 1.6.  Moreover $\omega \notin \fG_{r,s}$ by Proposition 1.5 since  $r+s \equiv~3~(mod~4)$. Hence $c(\omega) = \omega$ since $c(\omega) = \pm \omega$ by inspection. 
	
	2)  By Lemma 2.6 we have $\omega^{2} = (-1)^{(r+s)(r+s-1)/2}~ e_{1}^{2} e_{2}^{2} ... e_{r+s}^{2} = - e_{1}^{2} e_{2}^{2} ... e_{r+s}^{2} = (-1)^{r+1}$.

\begin{proof}  We prove 1).  If s is even, then r is odd and $\omega^{2} = 1$.  Following Proposition 3.5, chapter I of [LM] we define $\pi^{+} = \frac{1}{2} (1+ \omega)$ and  $\pi^{-} = \frac{1}{2} (1- \omega)$.  Note that both $\pi^{+}$ and $\pi^{-}$ lie in the center of $C\ell(r,s)$ since $\omega$ does.  Let $\fG_{1} = \pi^{+} \cdot \fG_{r,s}$ and $\fG_{2} = \pi^{-} \cdot \fG_{r,s}$.  We show that $\fG_{1}$ and $\fG_{2}$ have the desired properties.
\newline

\noindent We observed above that $c(\omega) = \omega$.  If $\xi \in \fG_{r,s}$, then $c(\omega \xi) = c(\xi) c(\omega) = - \xi \omega = - \omega \xi$.  It follows that  $\fG_{r,s}$ is invariant under left (or right) multiplication by $\omega$, and hence also by $\pi^{+}$ and $\pi^{-}$.  In particular $\fG_{i} \subset \fG_{r,s}$ for $i = 1,2$.  We conclude that $\fG_{r,s} = \fG_{1} \oplus \fG_{2}$, vector space direct sum, since $\pi^{+} + \pi^{-} = 1$.
\newline

\noindent  Note that $\fG_{1} \cdot \fG_{2} = \fG_{2} \cdot \fG_{1} = \{0 \}$ since $\pi^{+} \cdot \pi^{-} = \pi^{-} \cdot \pi^{+} = 0$.  Hence $[\fG_{1}, \fG_{2}] = \{0 \}$.  Note also that $\fG_{1}$ and $\fG_{2}$ are closed under the multiplication of $C\ell(r,s)$ since $\pi^{+} \cdot \pi^{+} = \pi^{+}$ and $\pi^{-} \cdot \pi^{-} = \pi^{-}$. Hence $\fG_{1}$ and $\fG_{2}$ are closed under left or right multiplication by elements of $C\ell(r,s)$.  It follows that $\fG_{i}$ is an ideal of $\fG_{r,s}$ for $i = 1,2$.
\newline

\noindent From Corollary 5.5 we know that $\fG_{r,s}^{\bc} = \fU^{\bc} \oplus \fU^{\bc}$ for some compact, simple Lie algebra $\fU$.  In each of the two cases $\fU^{\bc}$ is a complex, simple Lie algebra.  We also know that $\fG_{r,s}^{\bc} = \fG_{1}^{\bc} \oplus \fG_{2}^{\bc}$.  The decomposition of a complex, semisimple Lie algebra into a direct sum of simple ideals is unique up to order.  It follows that $\fG_{1}^{\bc} = \fG_{2}^{\bc} = \fU^{\bc}$, a complex simple Lie algebra.  This completes the proof of 1). 
\newline

\noindent We prove 2).  If s is odd, then r is even and $\omega^{2} = (-1)^{r+1} = -1$.  As above $\fG_{r,s}$ is invariant under left or right multiplication by $\omega$.  We may now define a complex vector space structure on $\fG_{r,s}$ by $(a + ib) \xi = a \xi + b \omega \xi$ for $a,b \in \br$ and $\xi \in \fG_{r,s}$.  The bracket operation on $\fG_{r,s}$ is $\bc$-bilinear since $ad~ \xi(\omega \eta) = \xi(\omega \eta) - (\omega \eta) \xi = \omega (\xi \eta - \eta \xi) = \omega ad~\xi(\eta)$.  Hence $\fG_{r,s}$ is a complex Lie algebra, and $\fG_{r,s}$ is semisimple over $\bc$ since it is semisimple over $\br$.
\end{proof}

\begin{lemma}  Let $\fH$ be a complex Lie algebra, and let $\fH_{0}$ denote $\fH$ regarded as a real Lie algebra.  Then $\fH_{0}^{\bc} \approx \fH \oplus \fH$.
\end{lemma}

\begin{proof}  See for example the proof of Theorem 2.51 of [W].
\end{proof}

\noindent  As we observed earlier in the proof of 1) $\fG_{r,s}^{\bc} = \fU^{\bc} \oplus \fU^{\bc}$ for some compact simple Lie algebra $\fU$.  We also noted that $\fU^{\bc}$ is simple as a complex Lie algebra.  Regarding $\fG_{r,s}$ as a complex Lie algebra we may write $\fG_{r,s} = \fG_{1} \oplus \fG_{2} \oplus ... \oplus \fG_{N}$, where $\{\fG_{1}, ... , \fG_{N} \}$ are the complex, simple ideals of $\fG_{r,s}$ (cf. Corollary 6.2, chapter II of [H]).   Regarding the ideals $\fG_{i}$ now as real Lie algebras it follows from Lemma 8.2 that $\fG_{i}^{\bc} \approx \fG_{i} \oplus \fG_{i}$ for all i.  Hence $\fG_{r,s}^{\bc}$ is the direct sum of 2N complex, simple ideals.  This is only possible if $N = 1$.  Hence $\fG_{r,s} = \fG_{1}$ is a complex simple Lie algebra.  The assertion 2) follows.
\newline

\noindent  The next result gives greater precision to Proposition 7.2 in this case.

\begin{proposition}  Let $r+s \equiv 3~(mod~4)$ and let s be even.  Let $\fG_{r,s} = \fG_{1} \oplus \fG_{2}$ be the decomposition of Proposition 8.1 such that $\fG_{1}^{\bc} = \fG_{2}^{\bc} = \fU^{\bc}$ for some simple compact Lie algebra $\fU$.  Then 

	1)  There exist involutive automorphisms $\tau_{1}$ and $\tau_{2}$ of $\fU$ such that $\tau_{\alpha}$ induces $\fG_{\alpha}$ for $\alpha = 1,2$.  Hence $\tau = \tau_{1} \times \tau_{2}$ induces $\fG_{r,s}$.
	
	2) $\fK_{r,s} \approx Fix(\tau) = Fix(\tau_{1}) \times Fix(\tau_{2})$.
\end{proposition}

\begin{proof} We begin the proof of the Proposition.  Let $\beta : \fG_{r,s} \rightarrow \fG_{r,s}$ be the automorphism constructed in Proposition 2.1 whose +1 eigenspace is $\fK_{r,s}$ and whose $-1$ eigenspace is $\fP_{r,s}$.  Note that $\beta(\omega) = (-1)^{s} \omega = \omega$ since s is even.  It follows that $\beta$ fixes $\pi^{+}$ and $\pi^{-}$ and leaves invariant $\fG_{1}$ and $\fG_{2}$.  Let $\beta_{\alpha}$ denote the restriction of $\beta$ to $\fG_{\alpha}$ for $\alpha = 1,2$.  Let $\fK_{\alpha},$ and $\fP_{\alpha}$ denote the $+1$ and $-1$ eigenspaces of $\beta_{\alpha}$ in $\fG_{\alpha} = \fK_{\alpha} \oplus \fP_{\alpha}$ and define a real Lie algebra $\fU_{\alpha} = \fK_{\alpha} \oplus i \fP_{\alpha} \subset \fG_{\alpha}^{\bc}$.  By Proposition 2.5 above and  Corollary 6.10 of [E] the Killing form of $\fG_{\alpha}$ is negative definite on $\fK_{\alpha}$ and positive definite on $\fP_{\alpha}$.  Hence the Killing form of $\fU_{\alpha}$ is negative definite on $\fU_{\alpha}$ for $\alpha = 1,2$.

\begin{lemma}  $\fU_{1}$ and $\fU_{2}$ are isomorphic to $\fU$ as Lie algebras, where $\fU^{\bc} = \fG_{\alpha}^{\bc}$ for $\alpha = 1,2$.  In particular $\fU_{1}$ and $\fU_{2}$ are compact simple Lie algebras.
\end{lemma}

\begin{proof} Note $\fU_{\alpha}^{\bc} = \fG_{\alpha}^{\bc} = \fU^{\bc}$ for $\alpha = 1,2$.  Hence both $\fU_{\alpha}$ and $\fU$ are compact real forms of $\fU^{\bc}$ and must be isomorphic (cf.  Corollary 7.3, chapter III of [H]).
\end{proof}

\noindent  We now conclude the proof of Proposition 8.3.  For $\alpha = 1,2$ define $\tau_{\alpha}$ to be the involutive automorphism of $\fU_{\alpha}$ whose $+1$ and $-1$ eigenspaces are $\fK_{\alpha}$ and i $\fP_{\alpha}$ respectively.  By inspection $\tau_{\alpha}$ induces $\fG_{\alpha}$ for $\alpha = 1,2$ and $Fix(\tau_{\alpha}) \approx \fK_{\alpha} = Fix(\beta_{\alpha})$.  Let $\tau = \tau_{1} \times \tau_{2}$, an involutive automorphism of $\fU \times \fU$ that induces $\fG = \fG_{1} \oplus \fG_{2}$. By Lemma 8.4 we may assume that $\fU_{1} = \fU_{2} = \fU$.  Finally, $\fK_{r,s} = Fix(\beta) = Fix(\beta_{1}) \times Fix(\beta_{2}) \approx Fix(\tau_{1}) \times Fix(\tau_{2}) = Fix(\tau)$.
\end{proof}

\section{Computation of $\fH_{r,s}$}

\noindent We consider only the cases where $r \geq 3$ and $s \geq 3$.  In each case $r+s = 8k+\alpha,~ 0 \leq \alpha \leq 7$ and $s \equiv \beta~(mod~4),~ 0 \leq \beta \leq 3$ we compute the dimensions of $\fK_{r,s}^{\prime}$ and $\fZ(\fK_{r,s}^{\prime})$ from Propositions 2.13 and 4.2, and Corollary 6.3 when $r+s \equiv~1~(mod~4)$.  By Propositions 7.2 and 8.3 these must equal the dimensions of $Fix(\tau)$ and $\fZ(Fix(\tau))$, where $\tau$ is the appropriate involutive automorphism of a compact Lie algebra $\fU$.  In each case only one automorphism $\tau$ in the Cartan list ([He], pp. 451-455) survives this comparison.
\newline

\noindent $\mathbf{Case~1 : r+s = 8k}$.

	a)  Note that $\fG_{r,s} = \fH_{r,s}$ by Propositions 1.6 and 1.7 and $\fK_{r,s} = \fK^{\prime}_{r,s}$ by the definition $\fK^{\prime}_{r,s} = \fK_{r,s} \cap \fH_{r,s}$.  From 1) of Proposition 2.13 we obtain dim $\fK_{r,s} = 2^{8k-2} - 2^{4k}~ 2^{-\frac{3}{2}}\{2^{-\frac{1}{2}} + cos(\frac{(2s-1) \pi}{4}) \}$.   In particular
	
\hspace{.2in} i) If $s \equiv 0~mod~4 $ or $s \equiv 1~mod~4$, then dim $\fK^{\prime}_{r,s} = 2^{8k-2} - 2^{4k-1}$ 

\hspace{.2in} ii) If $s \equiv 2~mod~4 $ or $s \equiv 3~mod~4$, then dim$\fK^{\prime}_{r,s}  = 2^{8k-2}$
\newline

	b)  From Propositions 4.1 and 4.2 we obtain 
	
\hspace{.2in} i) If $s \equiv 0~mod~4 $ or $s \equiv 1~mod~4$, then dim $\fK^{\prime}_{r,s}  = 0$.

\hspace{.2in} ii) If $s \equiv 2~mod~4 $ or $s \equiv 3~mod~4$, then dim $\fZ(\fK^{\prime}_{r,s} ) = 1$.
\newline

	c)  By Corollary 5.5  $\fG_{r,s}^{\bc} = \fU^{\bc}$, where $\fU = \fs \fo(2^{4k},\br)$.  By the discussion in section 7, $\fG_{r,s}$ is induced by an involutive automorphism $\tau$ of $\fU$.  See [He, pp. 451-455] for a precise statement.  By the discussion in [He] there are two conjugacy classes for $\tau$.
	
\hspace{.2in}  i)  BD I \hspace{.1in} 	$\fU= \fs \fo (p+q, \br)$, where $p+q = 2^{4k}$.

\noindent  Here $\fK_{r,s} \approx Fix(\tau) \approx \fs \fo(p,\br) \times  \fs \fo(q,\br)$.  We compute dim $\fK_{r,s} = \frac{1}{2}(p^{2} + q^{2}) - \frac{1}{2}(p+q) = (p-2^{4k-1})^{2} + 2^{8k-2} - 2^{4k-1}$.   

\hspace{.2in}  ii)  D III	 \hspace{.1in} $\fU= \fs \fo(2n,\br)$, where $n = 2^{4k-1}$.

\noindent  Here $\fK_{r,s} \approx Fix(\tau) \approx \fu (2^{4k-1})$ and dim $\fK_{r,s} = 2^{8k-2}$.
\newline

\noindent $\mathbf{Conclusion}$

\hspace{.2in} i) If $s \equiv 0~mod~4 $ or $s \equiv 1~mod~4$, then from  a) i) we have dim $\fK_{r,s} = 2^{8k-2} - 2^{4k-1}$, and this is compatible only with involutions of type BD I where $p = q = 2^{4k-1}$.  By the discussion in [He, pp. 451-455] we conclude that $\fH_{r,s} = \fs \fo(2^{4k-1}, 2^{4k-1})$.

\hspace{.2in} ii) If $s \equiv 2~mod~4 $ or $s \equiv 3 ~mod~4$, then from  a) ii) we have dim $\fK_{r,s} = 2^{8k-2}$, which is compatible only with involutions of type D III.  By the discussion in [He] we conclude that $\fH_{r,s} = \fs \fo^{*}(2^{4k})$. 
\newline

\noindent $\mathbf{Case~2 : r+s = 8k+1}$.

\noindent  Since $r+s \equiv 1~mod~4$ we see from Propositions 1.6 and 1.7 that $\fH_{r,s}$ is a codimension 1 ideal in $\fG_{r,s}$.  Moreover, from Proposition 6.1 we see that $dim~\fK_{r,s} = dim~\fK^{\prime}_{r,s}$ if $s = |\omega^{+}|$ is odd and $dim~\fK_{r,s} - 1 = dim~\fK^{\prime}_{r,s}$ if $s = |\omega^{+}|$ is even.

	a)  From Proposition 2.13 we obtain dim $\fK_{r,s} = 2^{8k-1} - 2^{4k-1}cos \frac{(s-1)\pi}{2}$.   From the discussion above we obtain
	
\hspace{.2in} i) If $s \equiv 0~mod~4 $ or $s \equiv 2~mod~4$, then dim $\fK^{\prime}_{r,s} = 2^{8k-1} - 1$ 

\hspace{.2in} ii) If $s \equiv 1~mod~4 $, then dim $\fK^{\prime}_{r,s} = 2^{8k-1} - 2^{4k-1}$

\hspace{.2in} iii) If $s \equiv 3~mod~4 $, then dim $\fK^{\prime}_{r,s} = 2^{8k-1} + 2^{4k-1}$
\newline

	b)  From Propositions 4.1 and 4.2 and Corollary 6.3  we obtain 
	
\hspace{.2in} i) If $s \equiv 0~mod~4 $ or $s \equiv 2~mod~4$, then dim $\fZ(\fK^{\prime}_{r,s}) = 1$.

\hspace{.2in} ii) If $s \equiv 1~mod~4 $ or $s \equiv 3~mod~4$, then dim $\fZ(\fK^{\prime}_{r,s}) = 0$.
\newline

	c)  By Propositions 1.6 and 1.7 and Corollary 5.5 we obtain $\fH_{r,s}^{\bc} = \fU^{\bc}$, where $\fU = \fs \fu(2^{4k})$.  Hence $\fH_{r,s}$ is a real form of $\fU^{\bc}$ and it is induced by an involutive automorphism $\tau$ of  $\fU$.  By the discussion in [He, pp. 451-455] there are three conjugacy classes for $\tau$ :
		
\hspace{.2in}  i)  A I \hspace{.1in} 	$\fU = \fs \fu (2^{4k})$

\noindent  Here $\fK_{r,s} \approx Fix(\tau) \approx \fs \fo(2^{4k})$.  We compute dim $\fK_{r,s} = 2^{8k-1} - 2^{4k-1}$.   

\hspace{.2in}  ii)  A II	 \hspace{.1in} $\fU= \fs \fu(2n)$, where $n = 2^{4k-1}$.

\noindent  Here $\fK_{r,s} \approx Fix(\tau) \approx \fs \fp(2^{4k-1})$.  We compute dim $\fK_{r,s} = 2^{8k-1} + 2^{4k-1}$

\hspace{.2in}  iii)  A III	 \hspace{.1in} $\fU= \fs \fu(p+q)$, where $p+q = 2^{4k}$.

\noindent  Here $\fK_{r,s} \approx Fix(\tau) \approx \{\left(\begin{array}{ccc}A & 0 \\ 0 & B\\ \end{array} \right) : A \in u(p), B \in u(q) ; trace~A + trace~B = 0 \}$. We compute $dim~\fK_{r,s} = p^{2} + q^{2} - 1 = 2(p-2^{4k-1})^{2} + 2^{8k-1} - 1$.
\newline

\noindent $\mathbf{Conclusion}$

\hspace{.2in} i) If $s \equiv 0~mod~4 $ or $s \equiv 2~mod~4$, then from  a) i) we have dim $\fK^{\prime}_{r,s} = 2^{8k-1} - 1$, and this is compatible only with involutions of type A III,  where $p = q = 2^{4k-1}$.  By the discussion in [He] we see that $\fH_{r,s} = \fs \fu(2^{4k-1}, 2^{4k-1})$.

\hspace{.2in} ii) If $s \equiv 1~mod~4 $, then from  a) ii) we have dim $\fK^{\prime}_{r,s} = 2^{8k-1} - 2^{4k-1}$, which is compatible only with involutions of type A I.  By the discussion in [He] we conclude that $\fH_{r,s} = \fs \ell (2^{4k},\br)$. 

\hspace{.2in} iii) If $s \equiv 3~mod~4 $, then from  a) iii) we have dim $\fK^{\prime}_{r,s} = 2^{8k-1} + 2^{4k-1}$, which is compatible only with involutions of type A II.  By the discussion in [He] we see that $\fH_{r,s} = \fs \fu^{*} (2^{4k})$. 
\newline

\noindent $\mathbf{Case~3 : r+s = 8k+2}$.

\noindent By Propositions 1.6 and 1.7 we see that $\fH_{r,s} = \fG_{r,s}$ and hence $\fK^{\prime}_{r,s} = \fK_{r,s}$.

	a)  From Proposition 2.13 we obtain dim $\fK^{\prime}_{r,s} = 2^{8k} + 2^{4k-1}  - 2^{4k}~2^{-\frac{1}{2}}cos \frac{(2s-3)\pi}{4}$.   From the discussion above we obtain
	
\hspace{.2in} i) If $s \equiv 0~mod~4 $ or $s \equiv 3~mod~4$, then dim $\fK^{\prime}_{r,s} = 2^{8k} + 2^{4k}$ 

\hspace{.2in} ii) If $s \equiv 1~mod~4 $ or $s \equiv 2~mod~4 $ then dim $\fK^{\prime}_{r,s} = 2^{8k}$
\newline

	b)  From Propositions 4.1 and 4.2 we obtain 
	
\hspace{.2in} i) If $s \equiv 0~mod~4 $ or $s \equiv 3~mod~4$, then dim $\fZ(\fK^{\prime}_{r,s}) = 0$.

\hspace{.2in} ii) If $s \equiv 1~mod~4 $ or $s \equiv 2~mod~4$, then dim $\fZ(\fK^{\prime}_{r,s}) = 1$.
\newline

	c)  By Corollary 5.5  $\fH_{r,s}^{\bc}  = \fG_{r,s}^{\bc} =\fU^{\bc} $, where $\fU = \fs \fp(2^{4k})$.  Hence $\fH_{r,s}$ is a real form of $\fU^{\bc}$ and it is induced by an involutive automorphism $\tau$ of  $\fU$.  By the discussion in [He] there are two conjugacy classes for $\tau$ :
		
\hspace{.2in}  i)  C I \hspace{.1in} 	$\fU = \fs \fp (2^{4k})$

\noindent  Here $\fK_{r,s} = \approx Fix(\tau) \approx \ \fu(2^{4k})$.  We compute dim $\fK_{r,s} = 2^{8k}$.   

\hspace{.2in}  ii)  C II	 \hspace{.1in} $\fU= \fs \fp(p+q)$, where $p+q = 2^{4k}$.

\noindent  Here $\fK_{r,s} \approx Fix(\tau) \approx \fs \fp(p) \times \fs \fp(q)$.  We compute dim $\fK_{r,s} = 2p^{2} + p + 2q^{2} + q = 2^{8k} + 2^{4k} + 4(p - 2^{4k-1})^{2}$.
\newline

\noindent $\mathbf{Conclusion}$

\hspace{.2in} i) If $s \equiv 0~mod~4 $ or $s \equiv 3~mod~4$, then from  a) i) we have dim $\fK^{\prime}_{r,s} = 2^{8k} + 2^{4k}$, and this is compatible only with involutions of type C II,  where $p = q = 2^{4k-1}$.  By the discussion in [He] we see that $\fH_{r,s} = \fs \fp(2^{4k-1}, 2^{4k-1})$.

\hspace{.2in} ii) If $s \equiv 1~mod~4 $ or $s \equiv 2~mod~4 $, then from  a) ii) we have dim $\fK^{\prime}_{r,s} = 2^{8k}$, which is compatible only with involutions of type C I.  By the discussion in [He] we see that $\fH_{r,s} = \fs \fp (2^{4k},\br)$. 
\newline

\noindent $\mathbf{Case~4 : r+s = 8k+3}$.

\noindent By Propositions 1.6 and 1.7 we see that $\fH_{r,s} = \fG_{r,s}$ and hence $\fK^{\prime}_{r,s} = \fK_{r,s}$.  By Proposition 1.6, $\omega$ lies in the center of $C\ell(r,s)$, but by Proposition 1.5, $\omega$ does not lie in $\fG_{r,s}$.
\newline

\noindent This case and case 8, where $r+s = 8k+7$, are more complicated than the others and need the results of section 8.

	a)  From Proposition 2.13 we obtain $dim ~\fK^{\prime}_{r,s} = dim~ \fK_{r,s} = 2^{8k+1} + 2^{4k}  - 2^{4k}~cos \frac{(s-2)\pi}{2}$.   From the discussion above we obtain
	
\hspace{.2in} i) If $s \equiv 0~mod~4 $,  then $dim~\fK^{\prime}_{r,s} = 2^{8k+1} + 2^{4k+1}$ 

\hspace{.2in} ii) If $s \equiv 1~mod~4 $, or $s \equiv 3~mod~4 $ then $dim~\fK^{\prime}_{r,s} = 2^{8k+1} + 2^{4k}$

\hspace{.2in} iii) If $s \equiv 2~mod~4 $,  then $dim~\fK^{\prime}_{r,s} = 2^{8k+1}$ 
\newline

\noindent We treat the cases s odd and s even separately.  The next result determines $\fH_{r,s}$ in the case that s is odd.

\begin{proposition}  If s is odd, then $\fH_{r,s} \approx \fs \fp(2^{4k}, \bc)_{\br}$.  
\end{proposition}

\begin{proof} If s is odd, then r is even. By remark 2) following Proposition 8.1 we see that $\omega^{2} = -1$, and $\omega$ defines a complex structure on $\fH_{r,s}$ such that $\fH = \fH_{r,s}$ is a complex simple Lie algebra.  By Lemma 8.2 we know that $\fH_{r,s}^{\bc} = \fH \oplus \fH$.  By Corollary 5.5 we also know that $\fH_{r,s}^{\bc} = \fG_{r,s}^{\bc} = \fU^{\bc} \oplus \fU^{\bc}$, where $\fU = \fs \fp(2^{4k})$.  By the uniqueness of the decomposition of $\fH_{r,s}^{\bc}$ into a direct sum of complex simple ideals it follows that $\fH = \fU^{\bc}$.  Hence $\fH_{r,s} = \fH_{\br} = (\fU^{\bc})_{\br} = \fs \fp(2^{4k}, \bc)_{\br}$. 
\end{proof}

\noindent  Next we treat the case that s is even.

	b)  From Propositions 4.1 and 4.2  we obtain 
	
\hspace{.2in} i) If $s \equiv 0~mod~4 $, then dim $\fZ(\fK^{\prime}_{r,s}) = 0$.

\hspace{.2in} ii) If $s \equiv 2~mod~4 $, then dim $\fZ(\fK^{\prime}_{r,s}) = 2$.
\newline

\noindent  By part 1) of Proposition 8.1 and Corollary 5.5 we obtain

\begin{lemma}  Let s be even.  Then $\fH_{r,s} = \fH_{1} \oplus \fH_{2}$, where $\fH_{1}^{\bc} = \fH_{2}^{\bc} = \fs \fp(2^{4k})^{\bc}$.
\end{lemma}

	c)    By the previous lemma $\fH_{1}$ and $\fH_{2}$  are real forms of $\fU^{\bc} = \fs \fp(2^{4k})^{\bc}$ and $\fH_{i}$ is determined by an automorphism $\tau_{i}$ of $\fU$ for $i = 1,2$.   By the discussion in [He] there are two conjugacy classes for $\tau$ :
		
\hspace{.2in}  i)  C I \hspace{.1in} 	$\fU = \fs \fp (2^{4k})$

\noindent  Here $\fK_{r,s} \approx Fix(\tau) \approx \ \fu(2^{4k})$.  We compute dim $\fK_{r,s} = 2^{8k}$ and dim $\fZ(\fK_{r,s}) = 1$. 

\hspace{.2in}  ii)  C II	 \hspace{.1in} $\fU= \fs \fp(p+q)$, where $p+q = 2^{4k}$.

\noindent  Here $\fK_{r,s} \approx Fix(\tau) \approx \fs \fp(p) \times \fs \fp(q)$.  We compute dim $\fK_{r,s} = 2p^{2} + p + 2q^{2} + q = 2^{8k} + 2^{4k} + 4(p - 2^{4k-1})^{2}$ and dim $\fZ(\fK_{r,s}) = 0$.
\newline

\noindent First we consider the case that $s \equiv 0~(mod~4)$.  By the discussion above in b) and Proposition 8.3 we obtain $0 = dim~\fZ(\fK_{r,s}) = dim~\fZ(Fix(\tau_{1})) + dim~\fZ(Fix(\tau_{2}))$, and it follows that $0 = dim~\fZ(Fix(\tau_{1})) = dim~\fZ(Fix(\tau_{2}))$.  Hence from the discussion in c) we conclude that $\tau_{1}$ and $\tau_{2}$ are both of type C II .  From a) above we have $dim~\fK_{r,s} = 2^{8k+1} + 2^{4k+1}$ and from c) we conclude that $dim~Fix(\tau_{1}) + dim~(Fix(\tau_{2}) = 2^{8k+1} + 2^{4k+1} + 4(p_{1} - 2^{4k-1})^{2} + 4(p_{2} - 2^{4k-1})^{2}$, where $p_{1} + q_{1} = p_{2} + q_{2} = 2^{4k}$.  This is only possible if $p_{\alpha} = q_{\alpha} = 2^{4k-1}$ for $\alpha = 1,2$.  We conclude that $\fH_{1} = \fH_{2} = \fs \fp(2^{4k-1}, 2^{4k_1})$, and $\fH_{r,s} = \fH_{1} \oplus \fH_{2} = \fs \fp(2^{4k-1}, 2^{4k-1}) \oplus \fs \fp(2^{4k-1}, 2^{4k-1})$.
\newline

\noindent  Finally we consider the case that $s \equiv 2~(mod~4)$.  By b) and Proposition 8.3 we have $2 = dim \fZ(\fK^{\prime}_{r,s}) = dim~\fZ(Fix(\tau_{1})) + dim~\fZ(Fix(\tau_{2}))$.  By c) we conclude that both $\tau_{1}$ and $\tau_{2}$ are involutions of type C I.  From the description of type C I in [He] it follows that $\fH_{1} = \fH_{2} = \fs \fp(2^{4k}, \br)$.  Hence $\fH_{r,s} = \fH_{1} \oplus \fH_{2} =  \fs \fp(2^{4k}, \br) \oplus  \fs \fp(2^{4k}, \br)$.
\newline

\noindent $\mathbf{Conclusion}$

\hspace{.2in} i) If s is odd, then $\fG_{r,s} \approx \fs \fp(2^{4k}, \bc)_{\br}$. 

\hspace{.2in} ii)  If $s \equiv 0~(mod~4)$, then $\fG_{r,s} \approx \fs \fp(2^{4k-1}, 2^{4k-1}) \oplus \fs \fp(2^{4k-1}, 2^{4k-1})$.

\hspace{.2in} iii)  If $s \equiv 2~(mod~4)$, then $\fG_{r,s} \approx \fs \fp(2^{4k}, \br) \oplus  \fs \fp(2^{4k}, \br)$.
\newline

\noindent $\mathbf{Case~5 : r+s = 8k+4}$.

\noindent By Propositions 1.6 and 1.7 we see that $\fH_{r,s} = \fG_{r,s}$ and hence $\fK^{\prime}_{r,s} = \fK_{r,s}$.

	a)  From Proposition 2.13 we obtain dim $\fK_{r,s} = 2^{8k+2} + 2^{4k}  - 2^{4k}~2^{\frac{1}{2}}cos \frac{(2s-5)\pi}{4}$.  In particular
	
\hspace{.2in} i) If $s \equiv 0~mod~4 $ or $s \equiv 1~mod~4$, then dim $\fK^{\prime}_{r,s} = 2^{8k+2} + 2^{4k+1}$ 

\hspace{.2in} ii) If $s \equiv 2~mod~4 $, or $s \equiv 3~mod~4 $ then dim $\fK^{\prime}_{r,s} = 2^{8k+2}$
\newline

	b)  From Propositions 4.1 and 4.2 we obtain 
	
\hspace{.2in} i) If $s \equiv 0~mod~4 $ or $s \equiv 1~mod~4$, then dim $\fZ(\fK^{\prime}_{r,s}) = 0$.

\hspace{.2in} ii) If $s \equiv 2~mod~4 $ or $s \equiv 3~mod~4$, then dim $\fZ(\fK^{\prime}_{r,s}) = 1$.
\newline

	c)  By Corollary 5.5,  $\fH_{r,s}^{\bc}  = \fG_{r,s}^{\bc} =\fU^{\bc} $, where $\fU = \fs \fp(2^{4k+1})$.  Hence $\fH_{r,s}$ is a real form of $\fU^{\bc}$ and it is induced by an involutive automorphism $\tau$ of  $\fU$.  By the discussion in [He] there are two conjugacy classes for $\tau$ :
		
\hspace{.2in}  i)  C I \hspace{.1in} 	$\fU = \fs \fp (2^{4k+1})$

\noindent  Here $\fK_{r,s}  \approx Fix(\tau) \approx  \fu(2^{4k+1})$.  We compute dim $\fK_{r,s} = 2^{8k+2}$ and $dim~\fZ(\fK_{r,s}) = 1$   

\hspace{.2in}  ii)  C II	 \hspace{.1in} $\fU= \fs \fp(p+q)$, where $p+q = 2^{4k+1}$.

\noindent  Here $\fK_{r,s}  \approx Fix(\tau) \approx \fs \fp(p) \times \fs \fp(q)$.  We compute dim $\fK_{r,s} = 2p^{2} + p + 2q^{2} + q = 2^{8k+2} + 2^{4k+1} + 4(p - 2^{4k})^{2}$ and  $dim~\fZ(\fK_{r,s}) = 0$.
\newline

\noindent $\mathbf{Conclusion}$

\hspace{.2in} i) If $s \equiv 0~mod~4 $ or $s \equiv 1~mod~4$, then from  b) i) we have dim $\fZ(\fK^{\prime}_{r,s}) = 0$.  This is compatible only with involutions of type C II.  By ai) we have $2^{8k+2} + 2^{4k+1} = dim~\fK_{r,s}^{\prime} = 2^{8k+2} + 2^{4k+1} + 4(p - 2^{4k})^{2}$.  This is only possible if $p = q = 2^{4k}$.  Hence $\fH_{r,s} = \fs \fp(2^{4k}, 2^{4k})$.

\hspace{.2in} ii) If $s \equiv 2~mod~4 $ or $s \equiv 3~mod~4 $, then from b) we have $dim~\fZ(\fK^{\prime}_{r,s}) = 1$.   This is only compatible with involutions of type CI, and we conclude that $\fH_{r,s} = \fs \fp (2^{4k+1},\br)$. 
\newline

\noindent $\mathbf{Case~6 : r+s = 8k+5}$.

\noindent  Since $r+s \equiv 1~mod~4$ we see from Propositions 1.6 and 1.7  that $\fH_{r,s}$ is a codimension 1 ideal in $\fG_{r,s}$.  Moreover, from Proposition 6.1 we see that $dim~\fK_{r,s} = dim~\fK^{\prime}_{r,s}$ if $s = |\omega^{+}|$ is odd and $dim~\fK_{r,s} - 1 = dim~\fK^{\prime}_{r,s}$ if $s = |\omega^{+}|$ is even.

	a)  From Proposition 2.13  we obtain dim $\fK_{r,s} = 2^{8k+3} - 2^{4k+1}cos \frac{(s-3)\pi}{2}$.   From the discussion above we obtain
	
\hspace{.2in} i) If $s \equiv 0~mod~4 $ or $s \equiv 2~mod~4$, then dim $\fK^{\prime}_{r,s} = 2^{8k+3} -1$ 

\hspace{.2in} ii) If $s \equiv 1~mod~4 $, then dim $\fK^{\prime}_{r,s} = 2^{8k+3} + 2^{4k+1}$

\hspace{.2in} iii) If $s \equiv 3~mod~4 $, then dim $\fK^{\prime}_{r,s} = 2^{8k+3} - 2^{4k+1}$
\newline

	b)  From Propositions 4.1 and 4.2  we obtain 
	
\hspace{.2in} i) If $s \equiv 0~mod~4 $ or $s \equiv 2~mod~4$, then $\omega \in \fK_{r,s}$ and $dim~\fZ(\fK_{r,s}) = 2$.  Hence $dim~\fZ(\fK^{\prime}_{r,s}) = dim~\fZ(\fK_{r,s}) -1 = 1$ by 1) of Corollary 6.3.

\hspace{.2in} ii) If $s \equiv 1~mod~4 $ or $s \equiv 3~mod~4$, then $\omega \notin \fK_{r,s}$ and $dim~\fZ(\fK_{r,s}) = 0$.  Hence dim$\fZ(\fK^{\prime}_{r,s}) = dim~\fZ(\fK_{r,s}) = 0$ by 2) of Corollary 6.3.
\newline

	c)  By Corollary 5.5 $\fH_{r,s}^{\bc}  = \fU^{\bc}$, where $\fU = \fs \fu(2^{4k+2})$.  Hence $\fH_{r,s}$ is a real form of $\fU^{\bc}$ and it is induced by an involutive automorphism $\tau$ of  $\fU$.  By the discussion in [He] there are three conjugacy classes for $\tau$ :
		
\hspace{.2in}  i)  A I \hspace{.1in} 	$\fU = \fs \fu (2^{4k+2})$

\noindent  Here $\fK_{r,s} \approx Fix(\tau) \approx \fs \fo(2^{4k+2})$.  We compute dim $\fK_{r,s} = 2^{8k+3} - 2^{4k+1}$ and dim $\fZ(\fK_{r,s}) = 0$. 

\hspace{.2in}  ii)  A II	 \hspace{.1in} $\fU= \fs \fu(2n)$, where $n = 2^{4k+1}$.

\noindent  Here $\fK_{r,s} \approx Fix(\tau) \approx \fs \fp(2^{4k+1})$.  We compute dim $\fK_{r,s} = 2^{8k+3} + 2^{4k+1}$ and dim $\fZ(\fK_{r,s}) = 0$.

\hspace{.2in}  iii)  A III	 \hspace{.1in} $\fU= \fs \fu(p+q)$, where $p+q = 2^{4k+2}$.

\noindent  Here $\fK_{r,s} \approx Fix(\tau) \approx \{\left(\begin{array}{ccc}A & 0 \\ 0 & B\\ \end{array} \right) : A \in \fu(p), B \in \fu(q) ; trace~A + trace~B = 0 \}$.

\noindent We compute $dim~\fK_{r,s} = p^{2} + q^{2} - 1 = 2(p-2^{4k+1})^{2} + 2^{8k+3} - 1$.
\newline

\noindent $\mathbf{Conclusion}$

\hspace{.2in} i) If $s \equiv 0~mod~4 $ or $s \equiv 2~mod~4$, then from  a) i) we have dim $\fK^{\prime}_{r,s} = 2^{8k+3} - 1$.  This is compatible only with involutions of type A III,  where $p = q = 2^{4k+1}$.  By the discussion in [He] we conclude that $\fH_{r,s} = \fs \fu(2^{4k+1}, 2^{4k+1})$.

\hspace{.2in} ii) If $s \equiv 1~mod~4 $, then from  a) ii) we have dim $\fK^{\prime}_{r,s} = 2^{8k+3} + 2^{4k+1}$.  By inspection type A I is incompatible and A III is incompatible since $dim~\fK_{r,s}^{\prime}$ is odd in this case.   Hence the involution $\tau$ must be of type A II.  From the discussion in [He] we conclude that $\fH_{r,s} = \fs \fu^{*}(2^{4k+2})$.

\hspace{.2in} iii) If $s \equiv 3~mod~4 $, then from  a) iii) we have dim $\fK^{\prime}_{r,s} = 2^{8k+3} - 2^{4k+1}$.  By inspection type A II is incompatible, and A III is incompatible since $dim~\fK_{r,s}^{\prime}$ is odd in this case.   Hence $\tau$ is of type A I, and we obtain $\fH_{r,s} = \fs \ell(2^{4k+2}, \br)$.
\newline

\noindent $\mathbf{Case~7 : r+s = 8k+6}$.

\noindent By Propositions 1.6 and 1.7  we see that $\fH_{r,s} = \fG_{r,s}$ and hence $\fK^{\prime}_{r,s} = \fK_{r,s}$.

	a)  From Proposition 2.13 we obtain dim $\fK_{r,s} = 2^{8k+4} - 2^{4k+1}  - 2^{4k}~2^{\frac{3}{2}}cos \frac{(2s-7)\pi}{4}$.   From the discussion above we obtain
	
\hspace{.2in} i) If $s \equiv 0~mod~4 $ or $s \equiv 3~mod~4$, then dim $\fK^{\prime}_{r,s} = 2^{8k+4} - 2^{4k+2}$ 

\hspace{.2in} ii) If $s \equiv 1~mod~4 $ or $s \equiv 2~mod~4 $, then dim $\fK^{\prime}_{r,s} = 2^{8k+4}$
\newline

	b)  From Propositions 4.1 and 4.2  we obtain 
	
\hspace{.2in} i) If $s \equiv 0~mod~4 $ or $s \equiv 3~mod~4$, then dim $\fZ(\fK^{\prime}_{r,s}) = 0$.

\hspace{.2in} ii) If $s \equiv 1~mod~4 $ or $s \equiv 2~mod~4$, then dim $\fZ(\fK^{\prime}_{r,s}) = 1$.
\newline

	c)  By Corollary 5.5, $\fH_{r,s}^{\bc}  = \fG_{r,s}^{\bc} =\fU^{\bc} $, where $\fU = \fs \fo(2^{4k+3}, \br)$.  Hence $\fH_{r,s}$ is a real form of $\fU^{\bc}$ and it is induced by an involutive automorphism $\tau$ of  $\fU$.  By the discussion in [He] there are two conjugacy classes for $\tau$ :
		
\hspace{.2in}  i)  BD I \hspace{.1in} 	$\fU = \fs \fo (2^{4k+3})$

\noindent  Here $\fK_{r,s} \approx Fix(\tau) \approx  \fs \fo(p) \times \fs \fo(q)$, where $p+q = 2^{4k+3}$.  By inspection dim $\fZ(\fK_{r,s}) = 0$. We compute dim $\fK_{r,s} = \frac{1}{2}(p^{2} + q^{2}) - \frac{1}{2}(p+q)) = 2^{8k+4} - 2^{4k+2} + (p - 2^{4k+2})^{2}$.   

\hspace{.2in}  ii)  D III	 \hspace{.1in} $\fU= \fs \fo(2n)$, where $n = 2^{4k+2}$.

\noindent  Here $\fK_{r,s} = Fix(\tau) \approx \fu(2^{4k+2})$.  By inspection dim $\fZ(\fK_{r,s}) = 1$, and we compute dim $\fK_{r,s} = 2^{8k+4}$.
\newline

\noindent $\mathbf{Conclusion}$

\hspace{.2in} i) If $s \equiv 0~mod~4 $ or $s \equiv 3~mod~4$, then from  a) i) we have dim $\fK^{\prime}_{r,s} = 2^{8k+4} - 2^{4k+2}$. This is compatible only with involutions of type BD I,  where $p = q = 2^{4k+2}$.  By the discussion in [He] we conclude that $\fH_{r,s} = \fs \fo(2^{4k+2}, 2^{4k+2})$.

\hspace{.2in} ii) If $s \equiv 1~mod~4 $, or $s \equiv 2~mod~4 $, then from  a) ii) we have dim $\fK^{\prime}_{r,s} = 2^{8k+4}$ and from b ii) we have dim $\fK^{\prime}_{r,s} = 1$, which is compatible only with involutions of type D III.  By the discussion in [He] we conclude that $\fH_{r,s} = \fs \fo^{*} (2^{4k+3})$. 
\newline

\noindent $\mathbf{Case~8 : r+s = 8k+7}$.

\noindent By Propositions 1.6 and 1.7  $\fH_{r,s} = \fG_{r,s}$ and hence $\fK^{\prime}_{r,s} = \fK_{r,s}$.  Note that $\omega$ lies in the center of $C\ell(r,s)$ by Proposition 1.6 , but $\omega$ does not lie in $\fG_{r,s}$ by Proposition 1.5.
\newline

	a)  From Proposition 2.13 we obtain $dim ~\fK^{\prime}_{r,s} = dim~ \fK_{r,s} = 2^{8k+5} - 2^{4k+2} (1 +cos \frac{s \pi}{2}$.   From the discussion above we obtain
	
\hspace{.2in} i) If $s \equiv 0~mod~4 $,  then $dim~\fK^{\prime}_{r,s} = 2^{8k+5} - 2^{4k+3}$. 

\hspace{.2in} ii) If $s \equiv 1~mod~4 $ or $s \equiv 3~mod~4 $, then $dim~\fK^{\prime}_{r,s} = 2^{8k+5} - 2^{4k+2}$

\hspace{.2in} i) If $s \equiv 2~mod~4 $,  then $dim~\fK^{\prime}_{r,s} = 2^{8k+5}$ 
\newline

\noindent We treat the cases s odd and s even separately.  The next result determines $\fH_{r,s}$ in the case that s is odd.

\begin{proposition}  If s is odd, then $\fH_{r,s} \approx \fs \fo(2^{4k+3}, \bc)_{\br}$.  
\end{proposition}

\begin{proof} If s is odd, then r is even.  By Remark 2) following Proposition 8.1 it follows that $\omega^{2} = -1$, and $\omega$ defines a complex structure on $\fH_{r,s}$ such that $\fH_{r,s} = \fH$ is a complex simple Lie algebra.  By Lemma 8.2 we know that $\fH_{r,s}^{\bc} = \fH \oplus \fH$.  By Corollary 5.5 we also know that $\fH_{r,s}^{\bc} = \fG_{r,s}^{\bc} = \fU^{\bc} \oplus \fU^{\bc}$, where $\fU = \fs \fo(2^{4k+3}, \br)$.  By the uniqueness of the decomposition of $\fH_{r,s}^{\bc}$ into a direct sum of complex simple ideals it follows that $\fH = \fU^{\bc}$.  Hence $\fH_{r,s} = \fH_{\br} = (\fU^{\bc})_{\br} = \fs \fo(2^{4k+3}, \bc)_{\br}$. 
\end{proof}

\noindent  Next we treat the case that s is even.

	b)  From Propositions 4.1 and 4.2 we obtain 
	
\hspace{.2in} i) If $s \equiv 0~mod~4$, then dim $\fZ(\fK^{\prime}_{r,s}) = 0$.

\hspace{.2in} i) If $s \equiv 2~mod~4 $, then dim $\fZ(\fK^{\prime}_{r,s}) = 2$.
\newline

\noindent  By part 1) of Proposition 8.1 and Corollary 5.5 we obtain

\begin{lemma}  Let s be even.  Then $\fH_{r,s} = \fH_{1} \oplus \fH_{2}$, where $\fH_{1}^{\bc} = \fH_{2}^{\bc} = \fs \fo(2^{4k+3},\bc)$.
\end{lemma}

	c)    By the previous lemma $\fH_{1}$ and $\fH_{2}$  are real forms of $\fU^{\bc} = \fs \fo(2^{4k+3}, \bc)$ and $\fH_{i}$ is determined by an automorphism $\tau_{i}$ of $\fU = \fs \fo(2^{4k+3}, \br)$ for $i = 1,2$.   By the discussion in [He] there are two conjugacy classes for $\tau$ :
		
\hspace{.2in}  i)  BD I \hspace{.1in} 	$\fU = \fs \fo (p+q,\br)$, where $p+q = 2^{4k+3}$.

\noindent  Here $\fK_{r,s} \approx Fix(\tau) \approx \fs \fo(p, \br) \times \fs \fo(q,\br))$.  By inspection $\fZ(\fK_{r,s}) = \{0 \}$ and a routine computation shows that  dim $\fK_{r,s} = \frac{1}{2}(p^{2} + q^{2}) - \frac{1}{2}(p+q) =  2^{8k+4} - 2^{4k+2} + (p - 2^{4k+2})^{2}$. 

\hspace{.2in}  ii)  D III	 \hspace{.1in} $\fU= \fs \fo(2n, \br)$, where $n = 2^{4k+2}$.

\noindent  Here $\fK_{r,s} \approx Fix(\tau) \approx \fu(2^{4k+2})$.  By inspection $dim~\fZ(\fK_{r,s}) = 1$ and  dim $\fK_{r,s} = 2^{8k+4}$.
\newline

\noindent First we consider the case that $s \equiv 0~(mod~4)$.  We recall from Propositions 7.2 and 8.3 that $dim~\fZ(\fK_{r,s}) = dim~\fZ(Fix(\tau_{1})) + dim~\fZ(Fix(\tau_{2}))$.  It follows from bi) that $dim~\fZ(\fK_{r,s}) = 0$ and we conclude that $dim~\fZ(Fix(\tau_{1})) = dim~\fZ(Fix(\tau_{2})) = 0$.  Hence both $\tau_{1}$ and $\tau_{2}$ are of type BD I by the discussion in c).  This shows that  $2^{8k+5} - 2^{4k+3} = dim~\fK_{r,s} = dim~Fix(\tau_{1})  + dim~Fix(\tau_{2})  = 2^{8k+5} - 2^{4k+3} + (p_{1} - 2^{4k+2})^{2} + (p_{2} - 2^{4k+2})^{2} $, where $p_{1} + q_{1} = p_{2} + q_{2} = 2^{4k+3}$.  We conclude that $p_{1} = q_{1} = p_{2} = q_{2} = 2^{4k+2}$.  Hence $\fH_{1} = \fH_{2} = \fs \fo(2^{4k+2}, 2^{4k+2})$ and $\fH_{r,s} = \fs \fo(2^{4k+2}, 2^{4k+2}) \oplus \fs \fo(2^{4k+2}, 2^{4k+2})$.
\newline

\noindent  Next  we consider the case that $s \equiv 2~(mod~4)$.  By bii), Proposition 7.2 and Proposition 8.3 we have $2 = dim~ \fZ(\fK^{\prime}_{r,s}) = dim~\fZ(Fix(\tau_{1})) + dim~\fZ(Fix(\tau_{2}))$.  By c) we conclude that both $\tau_{1}$ and $\tau_{2}$ are involutions of type D III.  From the description of type D III in [He] it follows that $\fH_{1} = \fH_{2} = \fs \fo^{*}(2^{4k+3})$.  Hence $\fH_{r,s} = \fH_{1} \oplus \fH_{2} =  \fs \fo^{*}(2^{4k+3}) \oplus  \fs \fo^{*}(2^{4k+3})$.
\newline

\noindent $\mathbf{Conclusion}$

\hspace{.2in} i) If s is odd, then $\fH_{r,s} \approx \fs \fo(2^{4k+3}, \bc)_{\br}$. 

\hspace{.2in} ii)  If $s \equiv 0~(mod~4)$, then $\fH_{r,s} \approx \fs \fo(2^{4k+2}, 2^{4k+2}) \oplus \fs \fo(2^{4k+2}, 2^{4k+2})$.

\hspace{.2in} iii)  If $s \equiv 2~(mod~4)$, then $\fH_{r,s} \approx \fs \fo^{*}(2^{4k+3}) \oplus  \fs \fo^{*}(2^{4k+3})$.
\newline

\noindent $\mathbf{References}$
\newline

\noindent [BCK], A. Benjamin, B. Chen and K. Kindred, ÒSums of evenly spaced binomial coefficientsÓ, Mathematics Magazine, vol. 83(5), December 2010, 370-373.
\newline

\noindent [E]  P. Eberlein, Ò Isometries of Clifford algebras, I Ó, arXiv : 1701.07421
\newline

\noindent [FH]  W. Fulton and J. Harris, ÒRepresentation Theory, A First CourseÓ, Springer, New York, 1991.
\newline

\noindent [Ha] F. R. Harvey, "Spinors and Calibrations", Perspectives in Mathematics vol.9, Academic Press, New York, 1990.
\newline

\noindent [He] S. Helgason, ÒDifferential Geometry, Lie groups and Symmetric SpacesÓ, Academic Press, New York, 1978.
\newline

\noindent [LM]  H.B. Lawson and M.-L.  Michelsohn, ÒSpin GeometryÓ, Princeton University Press, Princeton, 1989.
\newline

\noindent [W] F. Wisser, Ò Classification of complex and real semi simple Lie algebrasÓ, Diplomarbeit, Universitat Wien, June 2001, internet PDF.

\begin{tikzpicture}
\draw[xstep = 4cm, ystep = 2cm, color=blue] (-10,-10) grid (10,8) ;
\node at (-6.3,-8.4) {$\fs \fo(2^{4k+2},2^{4k+2})$} ;
\node at (-6.3,-8.9) {$\oplus$} ;
\node at (-6.3,-9.5) {$\fs \fo(2^{4k+2},2^{4k+2})$} ;
\node at (-2.3,-8.9){$\fs \fo(2^{4k+3},\bc)_{\br}$} ;
\node at (1.7, -8.4) {$\fs \fo^{*}(2^{4k+3})$} ;
\node at (1.7,-8.9){$\oplus$} ;
\node at (1.7, -9.5) {$\fs \fo^{*}(2^{4k+3})$} ;
\node at (5.7,-8.9){$\fs \fo(2^{4k+3},\bc)_{\br}$} ;
\node at (-9.4,-8.9){$r+s = 8k+7$} ;
\node at (-9.4,-6.9){$r+s = 8k+6$} ;
\node at (-6.2,-6.9) {$\fs \fo(2^{4k+2}, 2^{4k+2})$} ;
\node at (-2.2,-6.9) {$\fs \fo^{*}(2^{4k+3})$} ;
\node at (1.8,-6.9) {$\fs \fo^{*}(2^{4k+3})$} ;
\node at (5.8,-6.9) {$\fs \fo(2^{4k+2}, 2^{4k+2})$} ;
\node at (-9.4,-4.9){$r+s = 8k+5$} ;
\node at (-6.2,-4.9) {$\fs \fu(2^{4k+1}, 2^{4k+1})$} ;
\node at (-2.2,-4.9) {$\fs \fu^{*}(2^{4k+2})$} ;
\node at (1.8,-4.9) {$\fs \fu(2^{4k+1}, 2^{4k+1})$} ;
\node at (5.8,-4.9) {$\fs \ell(2^{4k+2}, \br)$} ;
\node at (-9.4,-2.9){$r+s = 8k+4$} ;
\node at (-6.2,-2.9) {$\fs \fp(2^{4k}, 2^{4k})$} ;
\node at (-2.2,-2.9) {$\fs \fp(2^{4k}, 2^{4k})$} ;
\node at (1.8,-2.9) {$\fs \fp(2^{4k+1}, \br)$} ;
\node at (5.8,-2.9) {$\fs \fp(2^{4k+1}, \br)$} ;
\node at (-9.4,-0.9){$r+s = 8k+3$} ;
\node at (-6, -0.5) {$\fs \fp (2^{4k-1},2^{4k-1})$} ;
\node at (-6,-1) {$\oplus$} ;
\node at (-6, -1.5) {$\fs \fp (2^{4k-1},2^{4k-1})$} ;
\node at (-2,-1) {$\fs \fp(2^{4k}, \bc)_{\br}$} ;
\node at (2, -0.5) {$\fs \fp (2^{4k},\br)$} ;
\node at (2,-1) {$\oplus$} ;
\node at (2, -1.5) {$\fs \fp (2^{4k},\br)$} ;
\node at (6,-1) {$\fs \fp(2^{4k}, \bc)_{\br}$} ;
\node at (-9.4,1.1){$r+s = 8k+2$} ;
\node at (-6,1.1){$\fs \fp(2^{4k-1},2^{4k-1})$} ;
\node at (-2,1.1){$\fs \fp(2^{4k},\br)$} ;
\node at (2,1.1){$\fs \fp(2^{4k},\br)$} ;
\node at (6,1.1){$\fs \fp(2^{4k-1},2^{4k-1})$} ;
\node at (-9.4,3.1){$r+s = 8k+1$} ;
\node at (-6,3.1){$\fs \fu(2^{4k-1}, 2^{4k-1})$} ;
\node at (-2,3.1) {$\fs \ell(2^{4k},\br)$} ;
\node at (2,3.1) {$\fs \fu(2^{4k-1},2^{4k-1})$} ;
\node at (6,3.1) {$\fs \fu^{*}(2^{4k})$} ;
\node at (-9.4,5.1){$r+s = 8k$} ;
\node at (-6,5.1){$\fs \fo(2^{4k-1}, 2^{4k-1})$} ;
\node at (-2,5.1){$\fs \fo(2^{4k-1}, 2^{4k-1})$} ;
\node at (2,5.1){$\fs \fo^{*}(2^{4k})$} ;
\node at (6,5.1){$\fs \fo^{*}(2^{4k})$} ;
\node at (-6,7.1){$s \equiv 0~(mod~4)$} ;
\node at (-2,7.1){$s \equiv 1~(mod~4)$} ;
\node at (2,7.1){$s \equiv 2~(mod~4)$} ;
\node at (6,7.1){$s \equiv 3~(mod~4)$} ;
\end{tikzpicture}
\newpage

\end {document}